\documentclass[12pt,dvipdfmx]{article}
\usepackage{amsmath,amsfonts,amssymb,amsthm}
\usepackage[dvipdfmx]{graphicx}
\usepackage[dvipdfmx]{color}
\newcommand{\red}{}
\usepackage{txfonts}%
\usepackage{longtable}
\usepackage{bm}
\usepackage{rotating}
\usepackage{diagbox}

\newcommand{\B}[1]{\mbox{\boldmath$#1$}} %
\newcommand{\R}[1]{\mbox{$\mathrm{#1}$}} %

\theoremstyle{plain}
\newtheorem{theorem}{Theorem}[section]

\theoremstyle{definition}

\theoremstyle{remark}
\newtheorem{remark}[theorem]{Remark}

\bmdefine{\Bz}{z}
\bmdefine{\Bx}{x}
\bmdefine{\Bt}{t}
\bmdefine{\Bh}{h}
\bmdefine{\Btheta}{\theta}
\bmdefine{\Beta}{\eta}

\allowdisplaybreaks[2]

\begin{document}

\title{Estimation of exponential-polynomial distribution by holonomic gradient descent}

\author{
Jumpei Hayakawa\thanks{Graduate School of Information Science and Technology, University of Tokyo}\ \ 
and Akimichi Takemura\footnotemark[1]%
}
\date{September, 2014}
\maketitle

\begin{abstract}
We study holonomic gradient decent for maximum likelihood estimation of
exponential-polynomial distribution, whose density is 
the exponential function of a polynomial in the random variable.  We
first consider the case that the support of the distribution is the set of positive reals. 
We show that the maximum likelihood estimate (MLE) can be easily computed by the holonomic gradient descent, even though the normalizing constant of this family does not have a closed-form expression and discuss 
determination of the degree of the polynomial based on the score test statistic.
Then we present extensions to the whole real line and to the bivariate distribution on the positive orthant.
\end{abstract}


\section{Introduction}
Exponential distribution and the truncated normal distribution have been frequently 
used for positive continuous random variables (e.g., Chapter 19 and Section 13.10 of \cite{N.L.Johnson}, \cite{F.Xu}). Generalizing these two cases, in this paper we consider fitting a density function which 
is the exponential function of a polynomial in the random variable.  For simplicity we 
first study the case of a positive random variable. For $x>0$, consider the following density
\begin{equation}
\label{eq:exp-pol}
f(x; \theta_1,\dots,\theta_d)= \frac{1}{A(\theta_1,\dots, \theta_d)} \exp(\theta_1 x + \dots + \theta_d x^d),
\qquad \theta_d < 0,
\end{equation}
where 
\begin{equation}
\label{eq:norm-constant}
A(\theta_1,\dots,\theta_d) = \int_0^\infty \exp(\theta_1 x + \dots + \theta_d x^d) \R{d}x
\end{equation}
is the normalizing constant of this density. In the following we write $A_d(\Btheta)=A(\theta_1,\dots,\theta_d)$.
We call \eqref{eq:exp-pol} the {\em exponential-polynomial distribution of order $d$.}
Although it is a natural generalization  of the exponential ($d=1$) and the truncated normal distribution ($d=2$), it has been
rarely used in statistics.  One reason is that $A_d(\Btheta)$ can not be written in a closed form.
Another reason may be that the tail of the distribution is light because of the term $\theta_d x^d$, 
$\theta_d < 0$. 
However by having this term, we can allow arbitrary values of $\theta_1,\dots,\theta_{d-1}$ and
have a flexible family of distributions.

Concerning the treatment of the normalizing constant, recently in \cite{H.Nakayama} we proposed
a new method, called the holonomic gradient decent (HGD), for evaluating the normalizing constant of
the exponential family and for computing MLE.  As in the subsequent works (\cite{H.Hashiguchi}, \cite{T.Sei}), 
we show that HGD works well also for the case of exponential-polynomial distribution.

When we fit \eqref{eq:exp-pol} to a given sample, the natural question we face is the
determination of the order $d$ of the model.  The exponential-polynomial model has a 
special structure that the model of order $d-1$ with $\theta_d=0$ and $\theta_{d-1}<0$ is the
boundary of the model of order $d$ with $\theta_d <0$. 
{\red In regular hypothesis testing problems or model selection problems, 
a submodel is assumed to be a smooth manifold of a smaller dimension in the {\em interior} of a larger model.}
Hence we need to adapt model selection procedures to this non-regular case.  We propose selection of $d$ by a score test.

The organization of this paper is as follows.   In Sections \ref{sec:mle}--\ref{sec:numerical}
we study exponential-polynomial distribution over the set of positive reals.
In Section  \ref{sec:mle} we derive a differential equation satisfied by $A_d(\Btheta)$ and use
the differential equation to compute MLE.
In Section \ref{sec:test} we discuss how to determine the order $d$ of the model by a score test.
In Section \ref{sec:numerical} we present results of some numerical experiments.
In Section \ref{sec:two-sided} we extend the  exponential-polynomial distribution to the
whole real line and in Section \ref{sec:bivariate-positive} we study a bivariate exponential-polynomial distribution. 
We end the paper with some discussions on further extension of the model in Section \ref{sec:discussion}.

\section{Maximum likelihood estimation via holonomic gradient descent}
\label{sec:mle}

Given a sample $\Bx=(x_1,\dots,x_n)$ of size $n$, $(1/n)$ times the 
log-likelihood function is written as
\begin{equation}
\label{eq:log-likelihood}
\bar l(\Btheta;\Bx)=\theta_1 \bar x + \theta_2 \bar{x^2} + \dots + \theta_d \bar{x^d} -  \psi(\Btheta), \qquad \psi(\Btheta)=\log A_d(\Btheta),
\end{equation}
where $\bar{x^m}=\sum_{i=1}^n x_i^m/n$, $m=1,\dots,d$.  
Let $\partial_m = \frac{\partial}{\partial \theta_m}$ denote the differentiation with respect to $\theta_m$.
In maximizing $\bar l$ with respect to $\Btheta$,
we want to compute its gradient 
\[
\nabla \bar l = 
\begin{bmatrix}\partial_1 \bar l \\ \vdots \\ \partial_d \bar l 
\end{bmatrix}
=
\begin{bmatrix}\bar x  \\ \vdots \\ \bar{x^d}
\end{bmatrix} 
- 
\begin{bmatrix}\partial_1  \psi\\ \vdots \\ \partial_d \psi 
\end{bmatrix}, \quad \partial_m \psi(\Btheta) = \frac{\partial_m A_d(\Btheta)}{A_d(\Btheta)}
\]
and its Hessian matrix 
\[
H(\bar l)(\Btheta)= - H(\psi)(\Btheta)=
- 
\begin{bmatrix}
\partial_1^2  \psi&  \cdots & \partial_1\partial_d \psi\\
\vdots &  \cdots & \vdots\\
\partial_d \partial_1 \psi& \cdots & \partial_d^2 \psi
\end{bmatrix}, \quad \partial_l \partial_m \psi(\Btheta)=\frac{\partial_l \partial_m A_d(\Btheta)}{A_d(\Btheta)}
- \frac{\partial_l A_d(\Btheta)}{A_d(\Btheta)} \frac{\partial_m A_d(\Btheta)}{A_d(\Btheta)}.
\]
Note that $I(\Btheta)=H(\psi)(\Btheta)$ is the Fisher information matrix for $\Btheta$.

In \eqref{eq:norm-constant} we can interchange the integration and the
differentiation by elements of $\Btheta$ as many time as needed.  Hence
derivatives of $A_d(\Btheta)$ can be evaluated by numerical integration.
However it is cumbersome to perform numerical integration for the 
derivatives at every $\Btheta$.  The holonomic gradient
decent allows us to compute $A_d(\Btheta)$ and its derivatives at any point by
numerically solving a differential equation from those at an initial point
$\Btheta=\Btheta^0$. 
The fact that $A_d(\Btheta)$ is a holonomic function 
(cf.\ Section 1 and Appendix of \cite{H.Nakayama}, Chapter 6 of \cite{hibi_book_13}, \cite{D.Zeilberger}) 
guarantees the existence of a differential equation with polynomial coefficients satisfied by $A_d(\Btheta)$.
Also, for  our problem there is a convenient initial point (see \eqref{eq:initial} below),
where $A_d(\Btheta)$ and its derivatives have a closed-form expression.  Hence by using the 
holonomic gradient descent, we do not need any numerical integration for our problem.

Differentiating  \eqref{eq:norm-constant} by $\theta_1$ we have
\[
\partial_1 A_d(\Btheta)=\int_0^\infty x \exp(\theta_1 x + \dots + \theta_d x^d) \R{d}x.
\]
Repeating this $i$ times we have
\begin{equation}
\label{eq:partial1i}
\partial_1^i A_d(\Btheta)=\int_0^\infty x^i \exp(\theta_1 x + \dots + \theta_d x^d) \R{d}x.
\end{equation}
However the right-hand side is also equal to $\partial_i A(\Btheta)$.  Hence 
the following relation holds.
\begin{equation}
\label{eq:diff1}
\partial_i A_d(\Btheta) = \partial_1^i A_d(\Btheta).
\end{equation}
In general, for any higher-order mixed derivative $\partial_1^{j_1} \dots \partial_d^{j_d} A(\Btheta)$ we have the relation
\[
\partial_1^{j_1} \dots \partial_d^{j_d} A_d(\Btheta)=\partial_1^{j_1 + 2 j_2 + \dots + d j_d}A_d(\Btheta).
\]
Hence all mixed derivatives reduce to the derivatives of $A_d(\Btheta)$ with respect to $\theta_1$.
It follows that for numerical purposes we only need to keep in memory the 
derivatives of $A_d(\Btheta)$ with respect to $\theta_1$.

Now as a relation among the derivatives of $A_d(\Btheta)$ with respect to $\theta_1$,  we have the following theorem.

\begin{theorem} \label{thm:1}
 $A_d(\Btheta)$ satisfies the following differential equation
\label{thm:diff-eq-1}
\begin{equation}
\label{eq:theorem1}
(\theta_1 + 2 \theta_2 \partial_1 + 3 \theta_3 \partial_1^2 + \dots + d \theta_d \partial_1^{d-1}) A_d(\Btheta)=-1.
\end{equation}
\end{theorem}

\begin{proof} 
\begin{align*}
-1 &= \big[ \exp(\theta_1 x + \dots + \theta_d x^d)\big]_0^\infty \\
  &= \int_0^\infty \partial_x \exp(\theta_1 x + \dots + \theta_d x^d) \; \R{d}x\\
   &= \int_0^\infty (\theta_1 + 2 \theta_2 x + 3 \theta_3 x^2 + \dots + d \theta_d x^{d-1}) \exp(\theta_1 x + \dots + \theta_d x^d)  \R{d}x \\
&= (\theta_1 + 2 \theta_2 \partial_1 + 3 \theta_3 \partial_1^2 + \dots + d \theta_d \partial_1^{d-1}) A_d(\Btheta).  \hspace{2cm} \text{(by \eqref{eq:partial1i})}
\end{align*}
\end{proof}

By \eqref{eq:theorem1}, $\partial_1^{d-1}A_d(\Btheta)$ is written in terms of lower-order derivatives as
\begin{equation}
\label{eq:solve-for-k-1}
\partial_1^{d-1}A_d(\Btheta)= 
{\red 
-\frac{1}{d\theta_d} \left\{1 + \Big(\theta_1 + 2 \theta_2 \partial_1 + 3 \theta_3 \partial_1^2 + \dots + (d-1) \theta_{d-1} \partial_1^{d-2}\Big) A_d(\Btheta)\right\}.
}
\end{equation}
Recursively differentiating this by $\theta_1$ we see that all higher-order derivatives
$\partial_1^m A_d(\Btheta)$, $m \ge d-1$, can be written in terms of the elements of a vector
\[
F(\Btheta)=[A_d(\Btheta),\partial_1 A_d(\Btheta),\dots, \partial_1^{d-2}A_d(\Btheta)]^{\mathsf{T}},
\]
where ${}^{\mathsf{T}}$ denotes the transpose of a vector or a matrix.
If $F(\Btheta)$ can be evaluated at any point $\Btheta$, then by \eqref{eq:diff1} we can evaluate
the gradient of $A_d$ and hence can compute MLE of the exponential-polynomial distribution.

The directional derivative of $F(\Btheta)$ in the direction  $\Bh=(h_1,\dots, h_d)$ is written as
\begin{equation}
\label{eq:pfaffian}
\frac{\partial}{\partial s} F(\Btheta + s \Bh)  = \sum_{j=1}^d h_j \partial_j F(\Btheta + s \Bh)  
= \sum_{j=1}^d h_j \partial_1^j F(\Btheta + s \Bh)  =
\sum_{j=1}^d h_j \begin{bmatrix} \partial_1^j A_d(\Btheta  + s \Bh)\\ \partial_1^{j+1} A_d(\Btheta  + s \Bh) \\ \vdots \\ \partial_1^{j+d-2} A_d(\Btheta  + s \Bh)
\end{bmatrix} .
\end{equation}
{\red 
If we differentiate \eqref{eq:solve-for-k-1} recursively, we obtain $(d-1)\times (d-1)$ matrices (called the Pfaffian matrices) 
$R_j$, $j=1,\dots,d$, with rational function elements such that the vector on the right-hand
side of \eqref{eq:pfaffian} is written as
\[
\begin{bmatrix} \partial_1^j A_d(\Btheta  + s \Bh)\\ \partial_1^{j+1} A_d(\Btheta  + s \Bh) \\ \vdots \\ \partial_1^{j+d-2} A_d(\Btheta  + s \Bh)
\end{bmatrix}
= R_j(\Btheta + s \Bh) F(\Btheta + s\Bh).
\]
Then the equation
\[
\frac{\partial}{\partial s} F(\Btheta + s \Bh) = 
\sum_{j=1}^d h_j R_j(\Btheta + s \Bh) F(\Btheta + s\Bh)
\]
can be solved by standard ODE solvers, such as the Runge-Kutta method, when 
an appropriate initial point $\Btheta_0$ and $F(\Btheta_0)$ are given. 
This is the procedure of HGD introduced in \cite{H.Nakayama}.
}

As a convenient initial point consider $\Btheta^0 = (0,0,\dots,0,-c)$, $c>0$. Then
\begin{equation}
\label{eq:initial}
\partial_1^m A_d(\Btheta^0)= \int_0^\infty x^m \exp(-cx^d) \R{d}x = \frac{1}{d} c^{-(1+m)/d} \Gamma\big(\frac{1+m}{d}\big), \qquad m\ge 0,
\end{equation}
which do not need numerical integration.

{\red
\begin{remark}
\label{rem:1}
Nobuki Takayama pointed out that $A_d$ satisfies an incomplete $A$-hypergeometric system 
introduced in \cite{nishiyama-takayama}.
In particular Theorem 2.2 of \cite{nishiyama-takayama} gives a basic result on incomplete
$A$-hypergeometric systems for a class of integrals including our $A_d$. 
See also Section 6.12 of \cite{hibi_book_13}. 
\end{remark}
}

In summary, we have shown that the evaluation of $A_d(\Btheta)$ and the maximization of the
likelihood function can be performed by using only a standard solver for an ordinary differential
equation.  As we see in Section \ref{sec:numerical} this method works quite well in practice.

\section{Determination of the degree of the model}
\label{sec:test}

When we fit the exponential-polynomial distribution in \eqref{eq:exp-pol} to a given sample, we need to 
determine the order $d$ of the model.  
Suppose that we are fitting the model with order $d-1$ and wondering  whether 
a model of order $d$ fits better.  One difficulty with 
\eqref{eq:solve-for-k-1} is that it becomes unstable as $\theta_d \rightarrow 0$, i.e., the
differential equation \eqref{eq:theorem1} has a singularity at $\theta_d=0$.  Hence if
the data really come from the model of order $d-1$, the estimation of the model of order $d$
by our method tends to be unstable.  

\begin{figure}[htbp]
\begin{center}
\includegraphics[width=4.5cm]{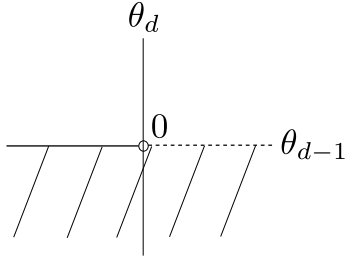}
\caption{Model of order $d-1$ within the model of order $d$}
\label{fig:model_d}
\end{center}
\end{figure}

We can understand this problem by considering the parameter spaces of order $d-1$ and $d$.
Let $\Omega_d=\{ (\theta_1, \dots, \theta_d) \mid \theta_d < 0\}\subset {\mathbb R}^d$ denote the
parameter space of the model of order $d$.  $\Omega_d$ is an open subset of ${\mathbb R}^d$.
Now $\Omega_{d-1}=\{ (\theta_1, \dots, \theta_{d-1},0) \mid \theta_{d-1} < 0\}$ 
considered as a subset of ${\mathbb R}^{d}$ is on the boundary of $\Omega_d$.
See Figure \ref{fig:model_d}. In $(\theta_{d-1},\theta_d)$-plane, $\Omega_d$ is the 
lower half open plane and $\Omega_{d-1}$ the left half open $\theta_{d-1}$-axis
$\{(\theta_{d-1},0)\mid \theta_{d-1} < 0\}$.
Since $A_d(\theta_1,\dots,\theta_{d-1},0)$ is finite for $\theta_{d-1}<0$, 
MLE may not exist in the open set $\Omega_d$ with positive probability.
For each $d$, consider  $\Omega_1,\dots, \Omega_d$  as subsets of ${\mathbb R}^d$ and let
$\bar\Omega_d=\Omega_1 \cup \cdots \cup \Omega_d$.  
{\red 
Since $A_d(\Btheta) < \infty$ if and only
if the last non-zero element of $\Btheta$ is negative, we have $\bar\Omega_d = \{ \Btheta \mid A(\Btheta) < \infty \}$.
}
$\psi_d(\Btheta) = {\red \log}  A_d(\Btheta)$  is strictly
convex on $\bar\Omega_d$ and approaches $+\infty$ as $\Btheta$ approaches the open boundary of $\bar\Omega_d$, such as the right half open $\theta_{d-1}$-axis $\{(\theta_{d-1},0)\mid \theta_{d-1} > 0\}$ in Figure 
\ref{fig:model_d}.  
Hence MLE always exists in $\bar\Omega_d$ but may not fall on $\Omega_d$.

We now consider the hypothesis testing problem:
\begin{equation}
\label{eq:testk-1}
H_0: \Btheta \in \Omega_{d-1} \ \ \text{v.s.}\ \ H_1: \Btheta\in \Omega_d .
\end{equation}
If $H_0$ is true %
let $\Btheta^*\in \Omega_{d-1}$ denote the true parameter vector and 
let $\hat\Btheta_{d-1} = (\hat\theta_1, \dots, \hat\theta_{d-1},0)$, $\hat\theta_{d-1}\le 0$, denote the
MLE under $H_0$.  {\red $\hat\Btheta_{d-1}$ may belong to $\Omega_k, k<d-1$. However as $n\rightarrow\infty$,
$\hat\Btheta_{d-1}$ converges to $\Btheta^*$ in probability.
}

The MLE  $\hat\Btheta_{d-1}$ under $H_0$ satisfies
\[
\partial_j \bar l(\hat\Btheta_{d-1};\Bx)=0, \qquad j=1,\dots, d-1.
\]
Note that $\bar l(\hat \Btheta_{d-1} + s \Bh;\Bx)$
is strictly concave in $s\ge 0$ for any $\Bh=(h_1,\dots,h_d)$, $h_d<0$, i.e., on
any half line emanating from $\hat\Btheta_{d-1}$ into $\Omega_d$.
Hence on this half line, $\bar l(\hat \Btheta_{d-1} {\red +  s \Bh};\Bx)$ is maximized at $s=0$ if and only
if 
\[
 \ 0\ge \frac{\partial}{\partial s} \bar l(\hat \Btheta_{d-1} + s \Bh;\Bx) |_{s=0} = 
\sum_{j=1}^d  h_j \partial_j \bar l(\hat\Btheta_{d-1};\Bx)= h_d \partial_d \bar l(\hat\Btheta_{d-1};\Bx) \\
 \Leftrightarrow  \ \partial_d \bar l(\hat\Btheta_{d-1};\Bx) \ge 0 .
\]
Note that the right-hand side does not depend on $\Bh$. 
{\red 
Hence MLE does not exist on $\Omega_d$ if and only if $\partial_d \bar l(\hat\Btheta_{d-1};\Bx) \ge 0$.   In this case $\hat\Btheta_{d-1}$ is the MLE over $\bar\Omega_d$.}

Let the $d\times d$ Fisher information matrix $I(\Btheta)=H(\psi)(\Btheta)$ be partitioned as
\[
I(\Btheta) = \begin{bmatrix} I_{d-1,d-1}(\Btheta) & I_{d-1,d}(\Btheta)\\
 I_{d,d-1}(\Btheta) & I_{dd}(\Btheta)
\end{bmatrix},
\]
where $I_{dd}$ is a scalar. 
{\red 
$I(\Btheta)$ is non-singular, since the score functions 
$\partial_m \bar l(\Btheta;\Bx)= \bar{x^m} - \partial_m \psi(\Btheta)$, $m=1,\dots,d$,
are moments and linearly independent for any $\Btheta$. 
}
Note that we put a comma between two subscripts when the subscripts are more complicated.
Define
\[
I_{dd\cdot 1,\dots,d-1}(\Btheta)=I_{dd}(\Btheta) -  I_{d,d-1}(\Btheta)  I_{d-1,d-1}(\Btheta)^{-1} I_{d-1,d}(\Btheta).
\]
In the standard case, where $\Omega_{d-1}$ is in the interior of $\Omega_d$, the two-sided test
based on 
\[
\frac{n(\partial_d \bar l(\hat\Btheta_{d-1};\Bx))^2}{I_{dd\cdot 1,\dots,d-1}(\hat\Btheta_{d-1})}
\]
is the score test for \eqref{eq:testk-1} (e.g., Section 7.7 of \cite{lehmann-large-sample}).
In our case $\Omega_{d-1}$ is the boundary of $\Omega_d$ and we reject $H_0$ if
$\partial_d\bar l(\hat\Btheta_{d-1};\Bx)$ is negative and its absolute value is too large.  
However, from the form of the log-likelihood function in \eqref{eq:log-likelihood}, 
the asymptotic null distribution $\partial_d\bar l(\hat\Btheta_{d-1};\Bx)$ is the
same as in the standard case, i.e.,
\begin{equation*}
\sqrt{n}\partial_d\bar l(\hat\Btheta_{d-1};\Bx) \stackrel{d}{\rightarrow} \R{N}(0,I_{dd\cdot 1,\dots,d-1}(\Btheta^*)) \qquad (n\rightarrow\infty).
\end{equation*}
Since $\hat\Btheta_{d-1}$ converges to $\Btheta^*$, we propose the following score test statistic
\begin{equation}
\label{eq:standard-asymptotic-theory-proposed-test}
T_{d-1}=\frac{\sqrt{n}\partial_d\bar l(\hat\Btheta_{d-1};\Bx)}{\sqrt{I_{dd\cdot 1,\dots,d-1}(\hat\Btheta_{d-1})}} .
\end{equation}
Let $z_\alpha$ denote the upper $\alpha$ quantile of $\R{N}(0,1)$. 
Given a significance level $\alpha < 1/2$, we can reject $H_0$ if
$T_{d-1} \le - z_\alpha$, 
in view of the convergence in distribution 
\begin{equation}
\label{eq:standard-asymptotic-theory-2}
T_{d-1}
 \stackrel{d}{\rightarrow} \R{N}(0,1) \qquad (n\rightarrow\infty).
\end{equation}

\section{Numerical experiments for the case of positive real line}
\label{sec:numerical}
We present results of some numerical experiments to show that MLE by
HGD
works well.  We also check the asymptotic approximation
in %
\eqref{eq:standard-asymptotic-theory-2}.

\subsection{Performance of MLE by the holonomic gradient descent}
The asymptotic distribution of MLE $\hat\Btheta_{d}$ is 
\begin{equation*}
\sqrt{n}(\hat\Btheta_{d}-\Btheta^*)
 \stackrel{d}{\rightarrow} \R{N}_{d}(0,I(\Btheta^*)^{-1}) \qquad (n\rightarrow\infty),
\end{equation*}
where
$I(\Btheta^*)$ is the Fisher information matrix at the true parameter  $\Btheta^*$. 
{\red
We assume that $\Btheta^*$ is an element of $\Omega_d$ (hence not an element of $\bar\Omega_d\setminus \Omega_d$).
}
Write 
\begin{equation}
\label{eq:mle-component-asymptotic-theory-define}
p_i = \frac{\sqrt{n}(\hat\theta_{i}-\theta_{i}^*)}{\sqrt{I^{-1}_{ii}(\Btheta^*)}}
,\qquad i=1,2,\dots, d,
\end{equation}
where  $I^{-1}_{ii}(\Btheta^*)$ denotes the $(i,i)$-component of $I(\Btheta^*)^{-1}$.
Then 
\begin{equation}
\label{eq:mle-component-asymptotic-theory}
p_i %
 \stackrel{d}{\rightarrow} \R{N}(0,1) \qquad (n\rightarrow\infty). %
\end{equation}
Thus in our experiments we fix 
the true parameter $\Btheta^*$, apply  our method to simulated samples %
many times and we check 
the convergence of the empirical distribution of $p_i$ 
to $\R{N}(0,1)$. %

\begin{figure}[thbp]
\begin{center}
\includegraphics[width=5cm]{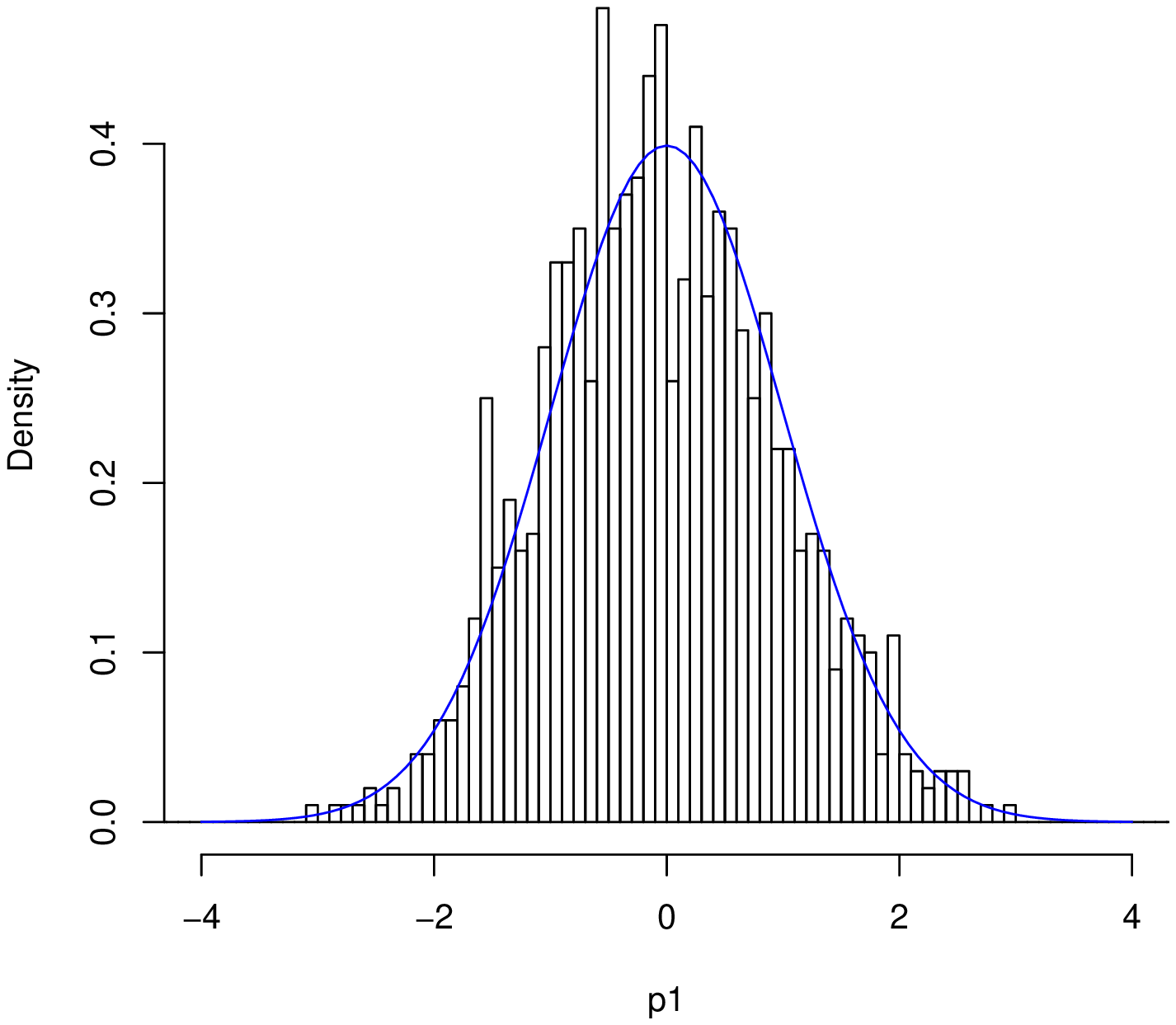}
\includegraphics[width=5cm]{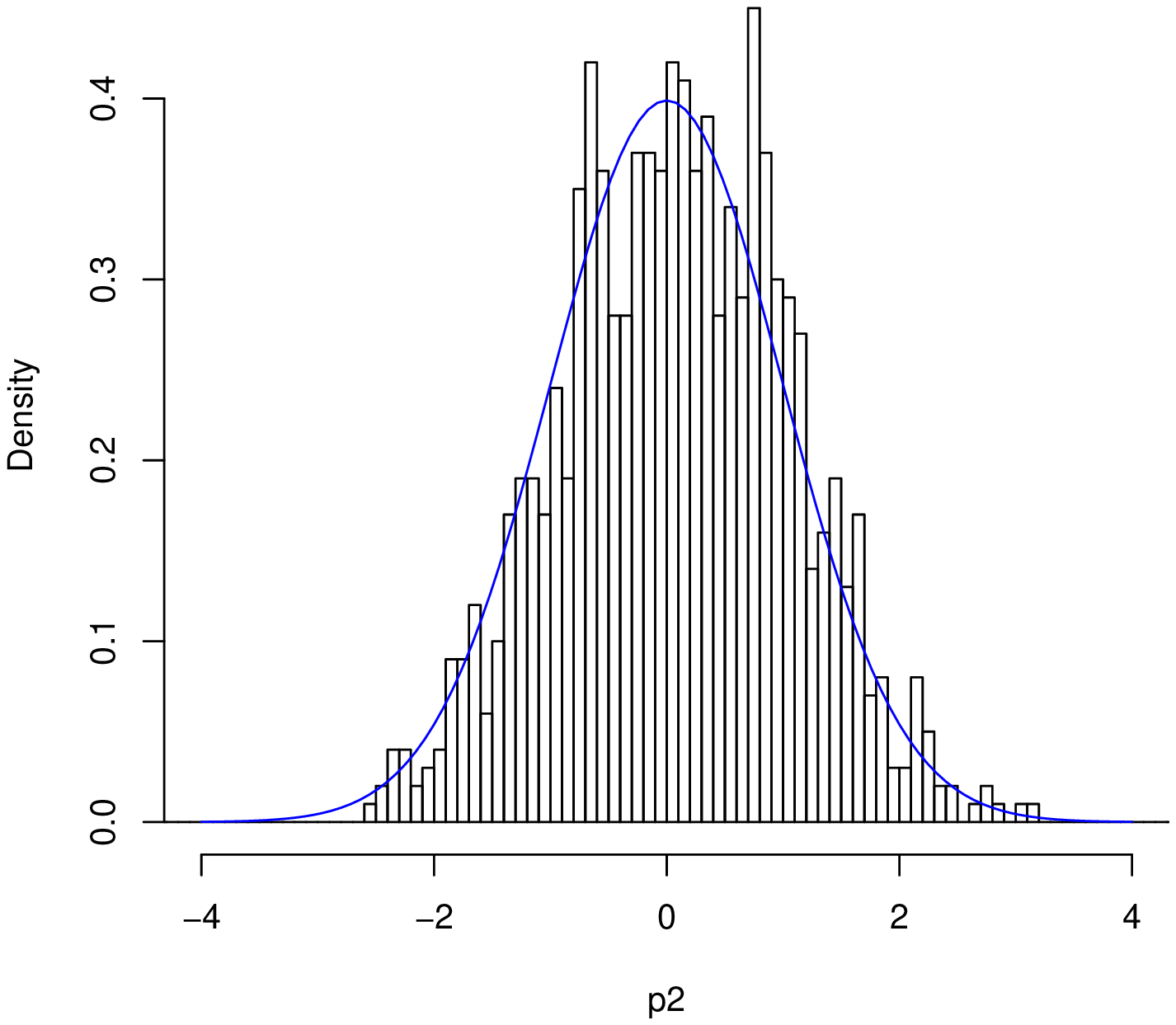}
\includegraphics[width=5cm]{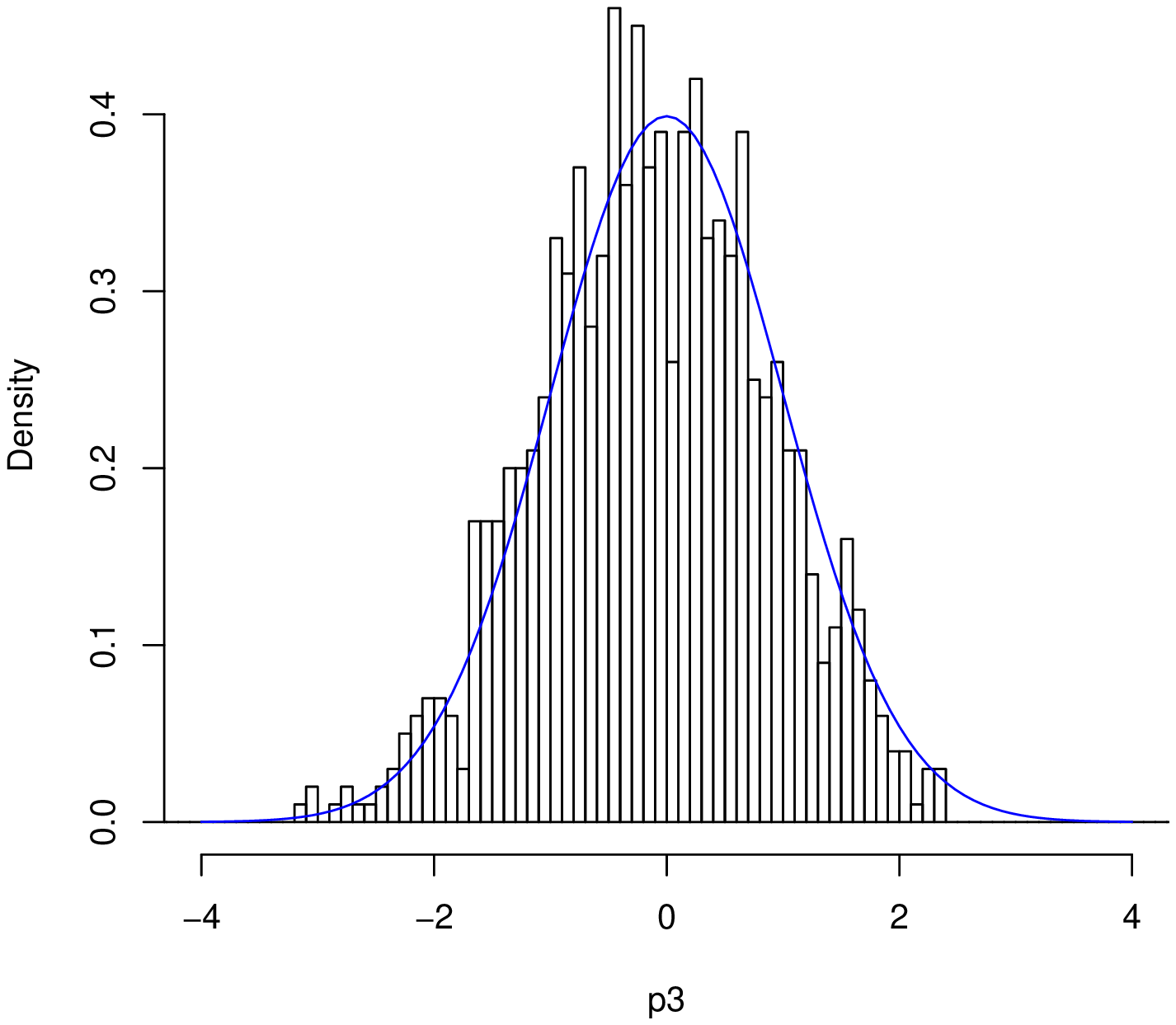}
\caption{Histogram of %
$p_i$, $i=1,2,3$ (from left to right) and density of \R{N}(0,1) for $d=3$ 
}
\label{fig:dim3_theta}
\end{center}
\end{figure}
We present simulation results %
for $d=3$ in \eqref{eq:exp-pol}. %
We set %
$\Btheta^*=(-1,3,-2)$. %
In the experiment we used $n=1000$ and  iterated computing MLE $1000$ times (i.e.\ the replication size is 1000).
Computation of MLE quickly converged in each iteration.
The histogram of %
$p_i$ %
is given  %
in Figure \ref{fig:dim3_theta}. 
The curved lines in these figures are the density function of $\R{N}(0,1)$. By comparing the histogram and the curved line we see that MLE by HGD works well.

\subsection{Asymptotic approximation for score tests}
\label{subsec:asymptotic_approximation}
We check the asymptotic approximation
in %
\eqref{eq:standard-asymptotic-theory-2} in the case of $d=3, 4$.
For $d=3$ we set $\Btheta^*=(3,-2,0)$. 
The histograms of $T_2$ and $T_3$ are shown in Figure \ref{fig:asymptotic} (left to right).
Again 
the asymptotic approximation works as expected.
\begin{figure}[htbp]
\begin{center}
\includegraphics[width=6cm]{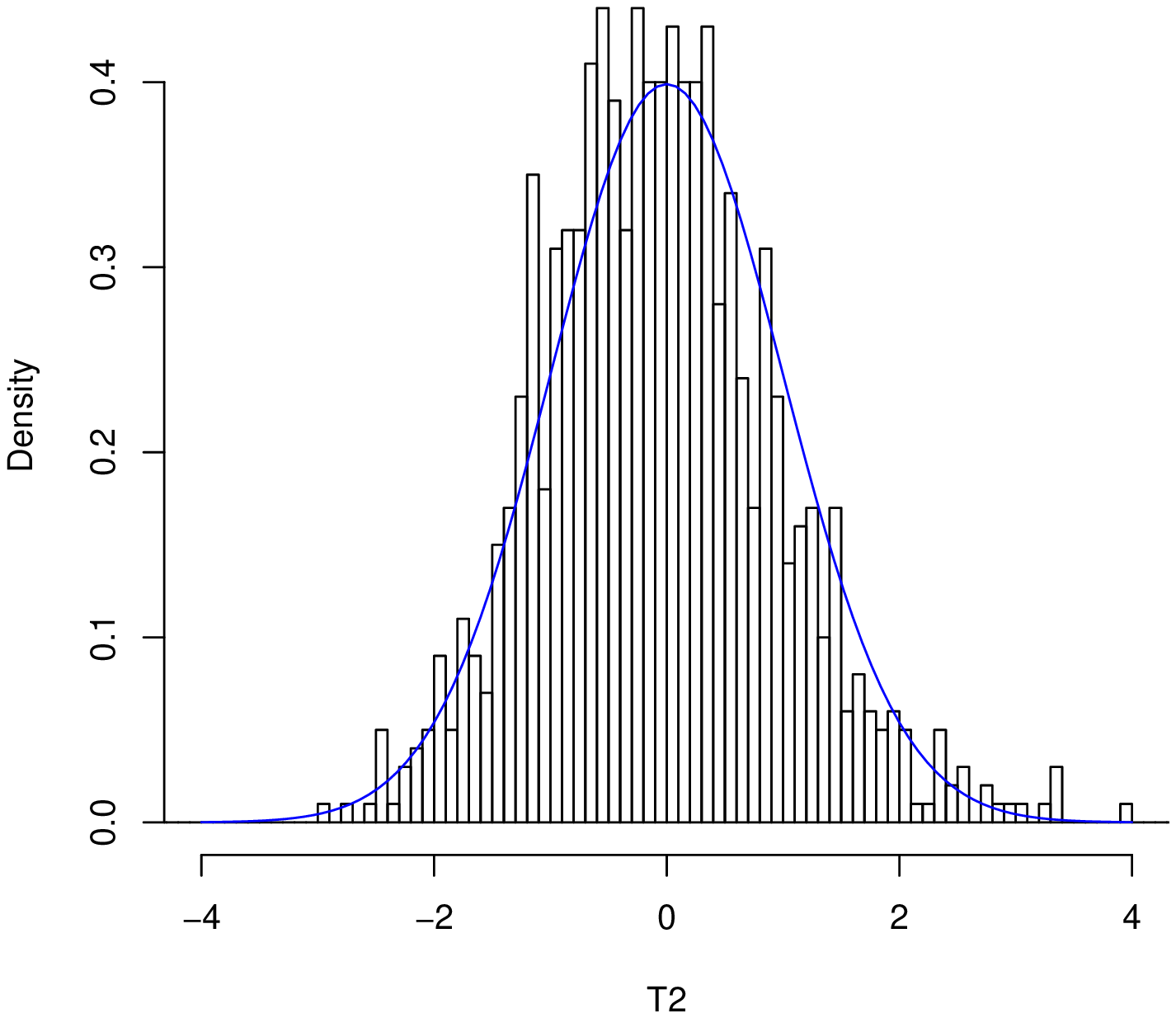}\qquad
\includegraphics[width=6cm]{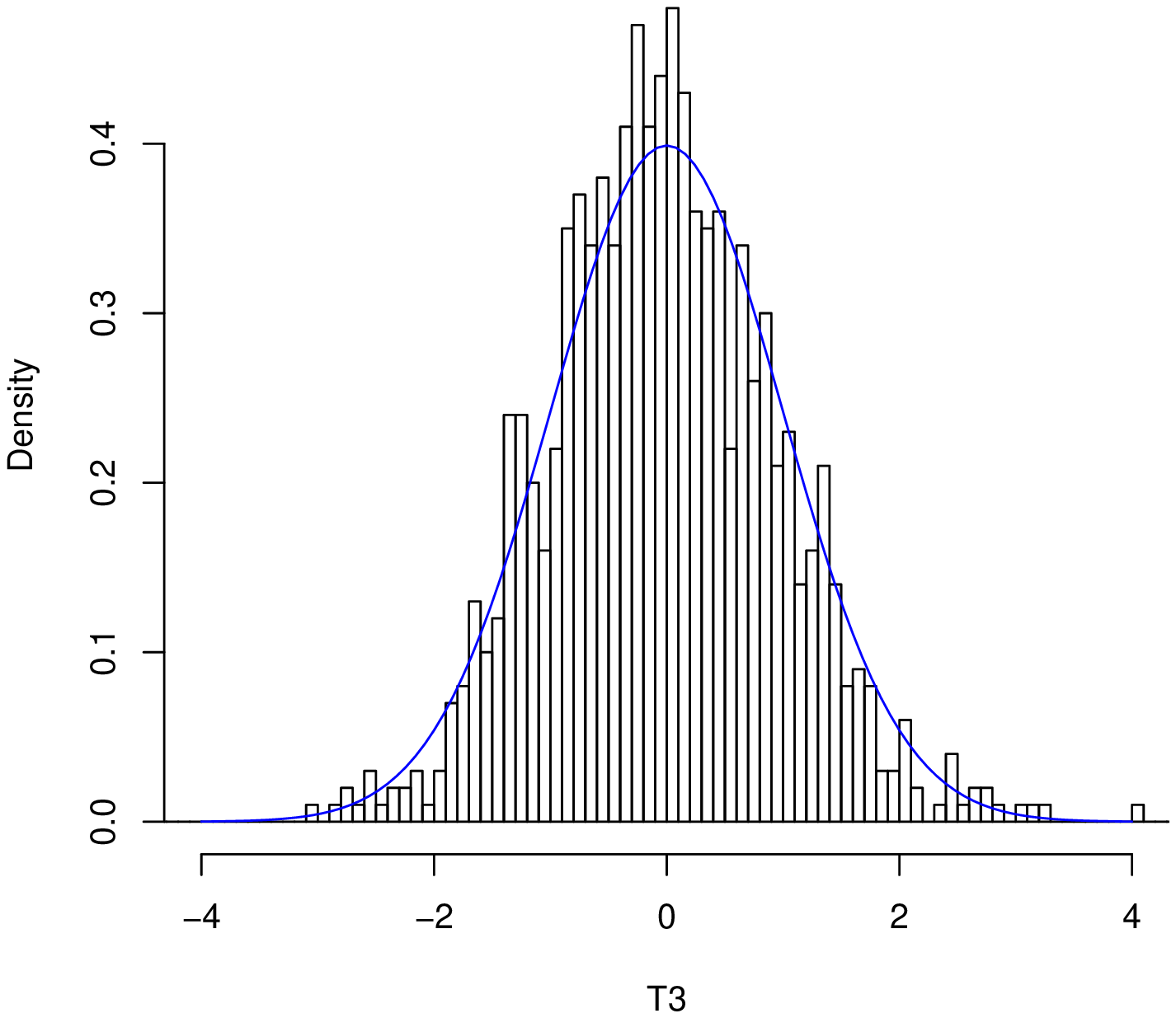}
\caption{Histogram of 
$T_{d-1}$ and the density of $\R{N}(0,1)$ for $d=3, 4$ 
}
\label{fig:asymptotic}
\end{center}
\end{figure}

\section{Exponential-polynomial distribution on  the whole real line}
\label{sec:two-sided}

In this section we extend the result of previous sections to the following density for the whole real line ${\mathbb R}^1$. Consider the density function 
\begin{equation}
\label{eq:exp-pol-whole}
f(x; \theta_1,\dots,\theta_{2d})= \frac{1}{A(\theta_1,\dots, \theta_{2d})} \exp(\theta_1 x + \dots + \theta_{2d} x^{2d}), 
\qquad \theta_{2d} < 0,
\end{equation}
where 
\begin{equation}
\label{eq:norm-constant-two-sided}
A(\theta_1,\dots,\theta_{2d}) = \int_{-\infty}^\infty \exp(\theta_1 x + \dots + \theta_{2d} x^{2d}) \R{d}x
\end{equation}
is the normalizing constant of this density. In following we write $A_{2d}(\Btheta)=A(\theta_1,\dots,\theta_{2d})$.

\subsection{Maximum likelihood estimation for the whole line}
The holonomic gradient decent is almost the same as in the previous sections.
We have
\[
\partial_1^i A_{2d}(\Btheta)=\int_{{\red -\infty}}^\infty x^i \exp(\theta_1 x + \dots + \theta_{2d} x^{2d}) \R{d}x, \quad i=1,2,\dots .
\]
Also %
$
\partial_i A_{2d}(\Btheta) = \partial_1^i A_{2d}(\Btheta).
$
In general %
$
\partial_1^{j_1} \dots \partial_d^{j_d} A_{2d}(\Btheta)=\partial_1^{j_1 + 2 j_2 + \dots + d j_d}A_{2d}(\Btheta).
$
Hence all mixed derivatives reduce to the derivatives of $A_{2d}(\Btheta)$ with respect to $\theta_1$.
It follows that for numerical purposes we only need to keep in memory the 
derivatives of $A_{2d}(\Btheta)$ with respect to $\theta_1$.

Now as a relation among the derivatives of $A_{2d}(\Btheta)$ with respect to $\theta_1$ we have the following theorem.

\begin{theorem}  $A_{2d}(\Btheta)$ satisfies the following differential equation
\begin{equation}
\label{eq:theorem2}
(\theta_1 + 2 \theta_2 \partial_1 + 3 \theta_3 \partial_1^2 + \dots + 2d \theta_{{\red 2d}} \partial_1^{2d-1}) A_{2d}(\Btheta)=0.
\end{equation}
\end{theorem}

Proof is omitted since it almost the same as the proof of Theorem \ref{thm:1}, by noting
\[0 = \big[ \exp(\theta_1 x + \dots + \theta_{2d} x^{2d})\big]_{-\infty}^\infty.\]
By \eqref{eq:theorem2}, $\partial_1^{2d-1}A_{2d}(\Btheta)$ is written in terms of lower-order derivatives as
\[
\partial_1^{2d-1}A_{2d}(\Btheta)= -\frac{1}{2d\theta_{2d}} (\theta_1 + 2 \theta_2 \partial_1 + 3 \theta_3 \partial_1^2 + \dots + (2d-1) \theta_{2d-1} \partial_1^{2d-2}) A_{2d}(\Btheta).
\]
Recursively differentiating this by $\theta_1$ all higher-order derivatives
$\partial_1^m A_{2d}(\Btheta)$, $m \ge 2d-1$, can be easily written in terms of
$A_{2d}(\Btheta),\partial_1 A_{2d}(\Btheta),\dots, \partial_1^{2d-2}A_{2d}(\Btheta)$.

As a convenient initial point consider $\Btheta^0 = (0,0,\dots,0,-c)$, $c>0$. Then
\[
\partial_1^m A_{2d}(\Btheta^0) =\int_{-\infty}^\infty x^m \exp(-cx^{2d}) \R{d}x = 
\begin{cases}\frac{1}{d} c^{-(1+m)/2d} \Gamma\big(\frac{1+m}{2d}\big)& \qquad m=0,2,4,\dots \\
0&\qquad m=1,3,5,\dots
\end{cases}
\]
which do not need numerical integration.

\subsection{Determination of the degree for the case of the whole line}
For determining the order of the model we consider the testing problem
\begin{equation*}
H_0: \Btheta \in \Omega_{2d-2} \ \ \text{v.s.}\ \ H_1: \Btheta\in \Omega_{2d} .
\end{equation*}
The parameter space is illustrated in Figure \ref{fig:model_2k}, where
$\Omega_{2d-2}$ corresponds to the origin.

\begin{figure}[htbp]
\begin{center}
\includegraphics[width=4.5cm]{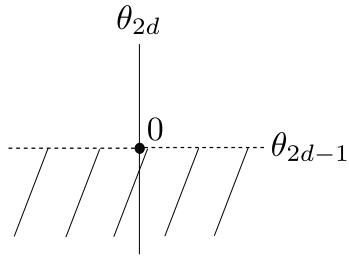}
\caption{Model of order $2d-2$ within the model of order $2d$}
\label{fig:model_2k}
\end{center}
\end{figure}

Here we need to do more careful analysis  than in Section \ref{sec:test}.
The difficulty in this case is that 
$A_{2d}(\Btheta)$ in \eqref{eq:norm-constant-two-sided} is infinite
for $\theta_{2d-1}\neq 0, \theta_{2d}=0$: 
\[
A(\theta_1,\dots,\theta_{2d},\theta_{2d-1},0) = \infty , \qquad \forall \theta_{2d-1}\neq 0.
\]
Hence we can not take the partial derivative of $A_{2d}(\Btheta)$ with respect to $\theta_{2d-1}$ at
$(\theta_1,\dots,\theta_{2d-2},0,0)$.
However 
$\partial_{2d-1} A(\theta_1,\dots,\theta_{2d})$ and $\partial_{2d} A(\theta_1,\dots,\theta_{2d})$
exist, as long as $\theta_{2d}< 0$.
Also if $\theta_{2d-2}<0$, 
as $(\theta_{2d-1},\theta_{2d})\rightarrow (0,0)$ in such a way that
$|\theta_{2d-1}/\theta_{2d}|$ is bounded,  
by the dominated convergence theorem we have
\begin{align*}
  \lim_{(\theta_{2d-1},\theta_{2d})\rightarrow (0,0) \atop |\theta_{2d-1}/\theta_{2d}|\,:\,\text{bounded}} 
\partial_{2d-1} A(\theta_1,\dots,\theta_{2d})
  &= \int_{-\infty}^\infty x^{2d-1} \exp(\theta_1 x + \dots + \theta_{2d-2} x^{2d-2}) \R{d}x \nonumber \\
&=A(\Btheta_{2d-2}) E_{\Btheta_{2d-2}}(X^{2d-1}),
\end{align*}
and
\[
\lim_{(\theta_{2d-1},\theta_{2d})\rightarrow (0,0) \atop |\theta_{2d-1}/\theta_{2d}|\,:\,\text{bounded}} 
\partial_{2d} A(\theta_1,\dots,\theta_{2d})
=A(\Btheta_{2d-2}) E_{\Btheta_{2d-2}}(X^{2d}),
\]
where $\Btheta_{2d-2}=(\theta_1,\dots,\theta_{2d-2},0,0)$, $\theta_{2d-2}<0$ and 
$E_{\Btheta_{2d-2}}$ denotes the expected value under  $\Btheta_{2d-2}$.
Let $\hat\Btheta_{2d-2}$ denote MLE under $H_0$.

We now redefine the $(2d)\times (2d)$ Fisher information matrix $I(\Btheta)$ at
$\Btheta_{2d-2}$ as
\[
\tilde I(\Btheta_{2d-2})= 
\begin{bmatrix} \tilde I_{2d-2,2d-2}(\Btheta_{2d-2}) & \tilde I_{2d-2,2d}(\Btheta_{2d-2}) \\
 \tilde I_{2d,2d-2}(\Btheta_{2d-2}) & \tilde I_{2d,2d}(\Btheta_{2d-2})
\end{bmatrix}
=\lim_{(\theta_{2d-1},\theta_{2d})\rightarrow (0,0) \atop |\theta_{2d-1}/\theta_{2d}|\,:\,\text{bounded}}
\begin{bmatrix} I_{2d-2,2d-2}(\Btheta) & I_{2d-2,2d}(\Btheta)\\
 I_{2d,2d-2}(\Btheta) & I_{2d,2d}(\Btheta)
\end{bmatrix},
\]
where $I_{2d,2d}$ is a $2\times 2$ matrix.  
{\red 
Since the elements of $\tilde I(\Btheta_{2d-2})$ are defined by the moments and any polynomial
function of $X$ is not degenerate, $\tilde I(\Btheta_{2d-2})$ is non-singular.
}
Define
\[
\tilde I_{2d,2d\cdot 1,\dots,2d-2}(\Btheta)=\tilde I_{2d,2d}(\Btheta) -  
\tilde I_{2d,2d-2}(\Btheta)  \tilde I_{2d-2,2d-2}(\Btheta)^{-1} \tilde I_{2d-2,2d}(\Btheta).
\]

For testing $H_0$ we again propose to use a score statistic
\begin{equation}
\label{eq:test-statistic-whole-line}
T_{2d-2}=
n[\partial_{2d-1} \bar l(\hat\Btheta_{2d-2};\Bx), \partial_{2d}\bar l(\hat\Btheta_{2d-2};\Bx)] \ 
\tilde I_{2d,2d\cdot 1,\dots,2d-2}(\hat\Btheta_{2d-2})^{-1}
\ \begin{bmatrix}\partial_{2d-1} \bar l(\hat\Btheta_{2d-2};\Bx) \\
\partial_{2d}\bar l(\hat\Btheta_{2d-2};\Bx)
\end{bmatrix}. %
\end{equation}
We reject $H_0$ if
\begin{equation}
T_{2d-2} 
\ge \chi^2_2(\alpha), 
\label{eq:dist-conv-whole1}
\end{equation}
where $\chi^2_2(\alpha)$ is the upper $\alpha$-quantile of the $\chi^2$ distribution with two degrees of freedom. Numerical performance of this test is confirmed in the next subsection.

\subsection{Numerical experiments for the whole line}
For checking the asymptotic distribution of the MLE, 
we compare the empirical distribution of 
$p_i$ in \eqref{eq:mle-component-asymptotic-theory-define}
with $\R{N}(0,1)$
for $2d=4$ and $\B{\theta}^{\ast} = (1,4,-2,-3)$.
For checking  \eqref{eq:dist-conv-whole1}
we compare the empirical distribution of $T_{2d-2}$ of \eqref{eq:test-statistic-whole-line}
with the $\chi^2$ distribution with 2 degrees of freedom.
For $2d-2=2$ we choose $\B{\theta}^{\ast} = (2,-1,0,0)$.

\begin{figure}[htbp]
\begin{center}
\includegraphics[width=3.9cm]{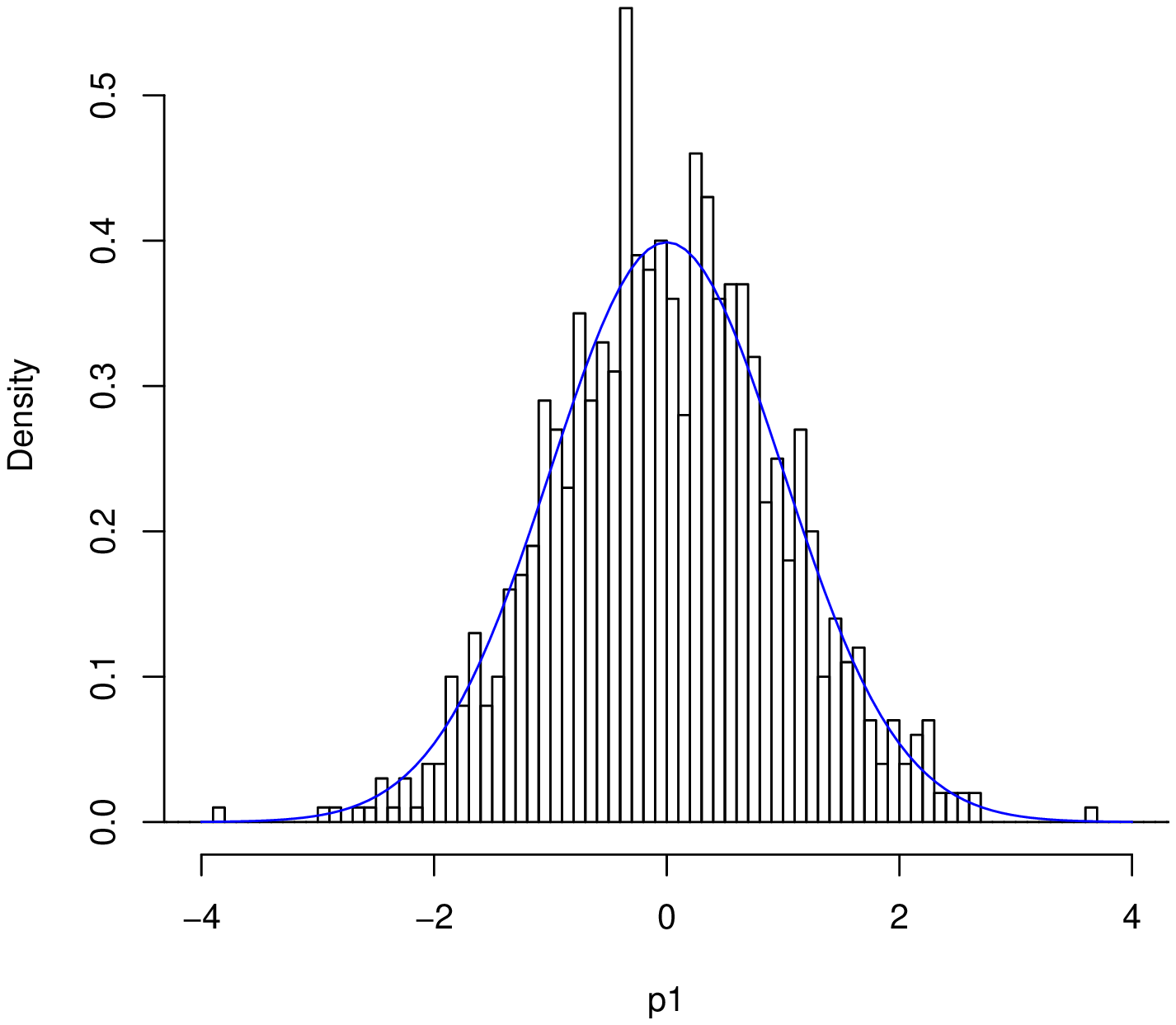}
\includegraphics[width=3.9cm]{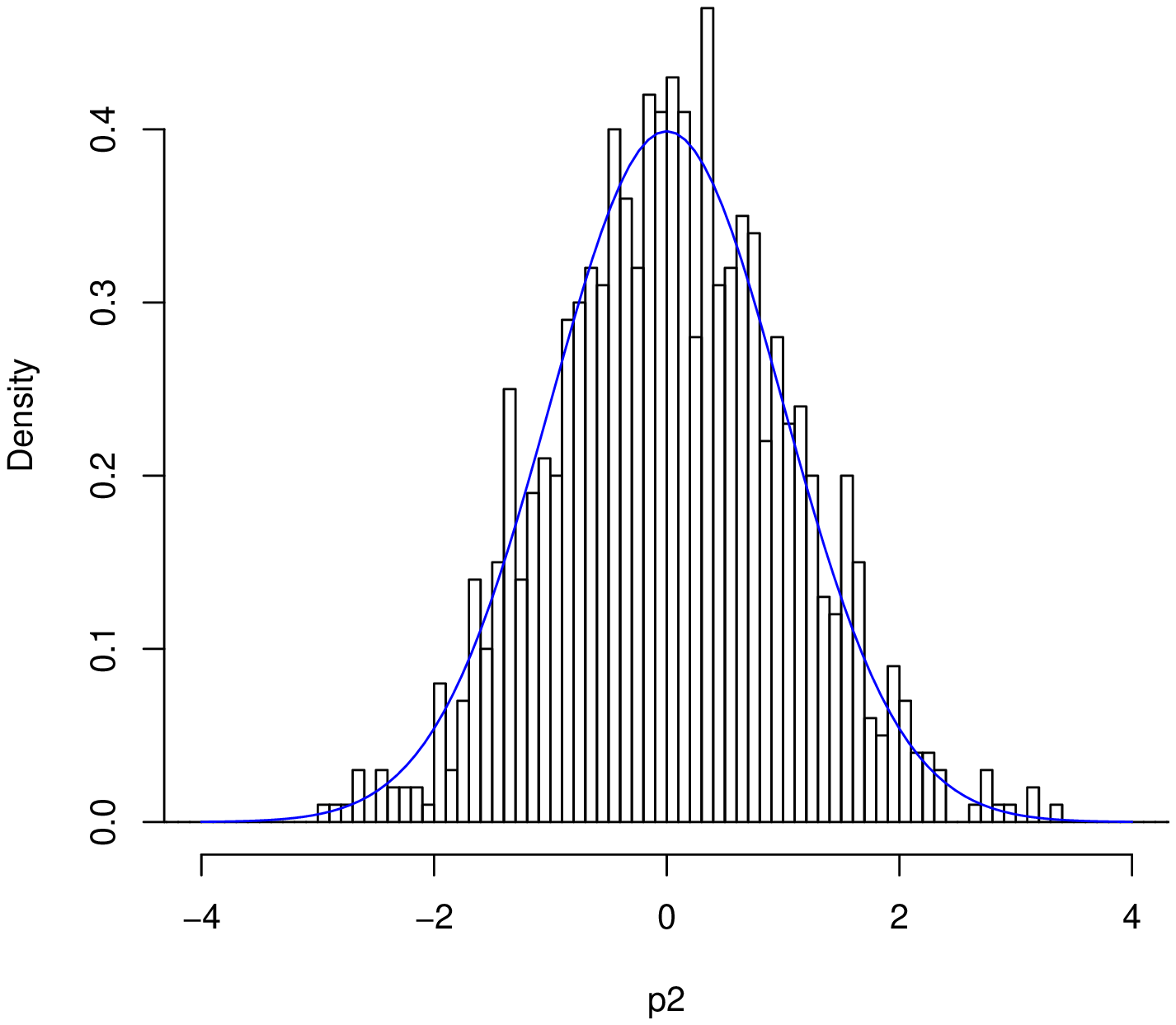}
\includegraphics[width=3.9cm]{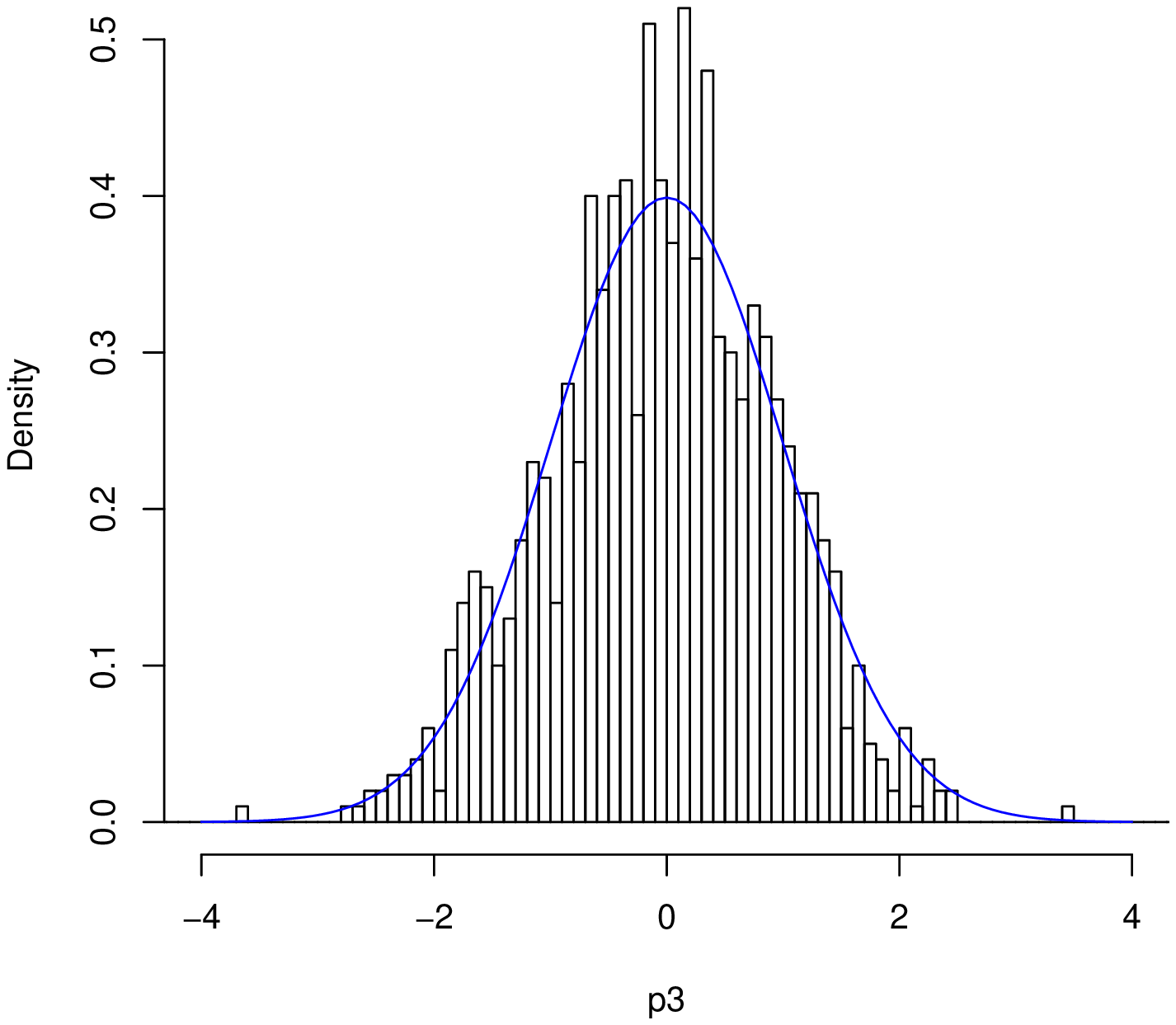}
\includegraphics[width=3.9cm]{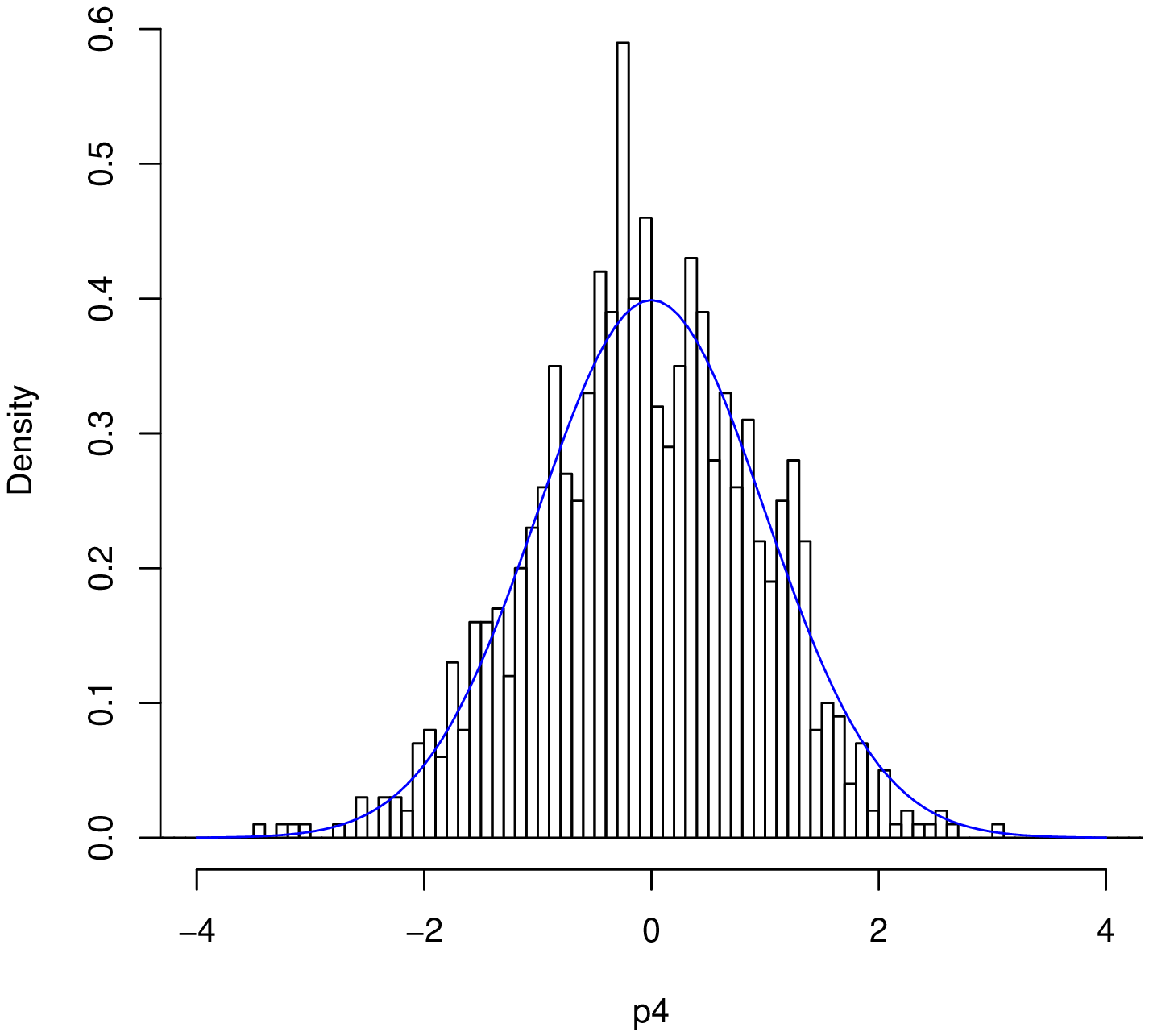}
\caption{Histogram $p_i$ and the density of $\R{N}(0,1)$ for $2d=4$%
}
\label{fig:whole_4}
\end{center}
\end{figure}
\begin{figure}[htbp]
\begin{center}
\includegraphics[width=5.5cm]{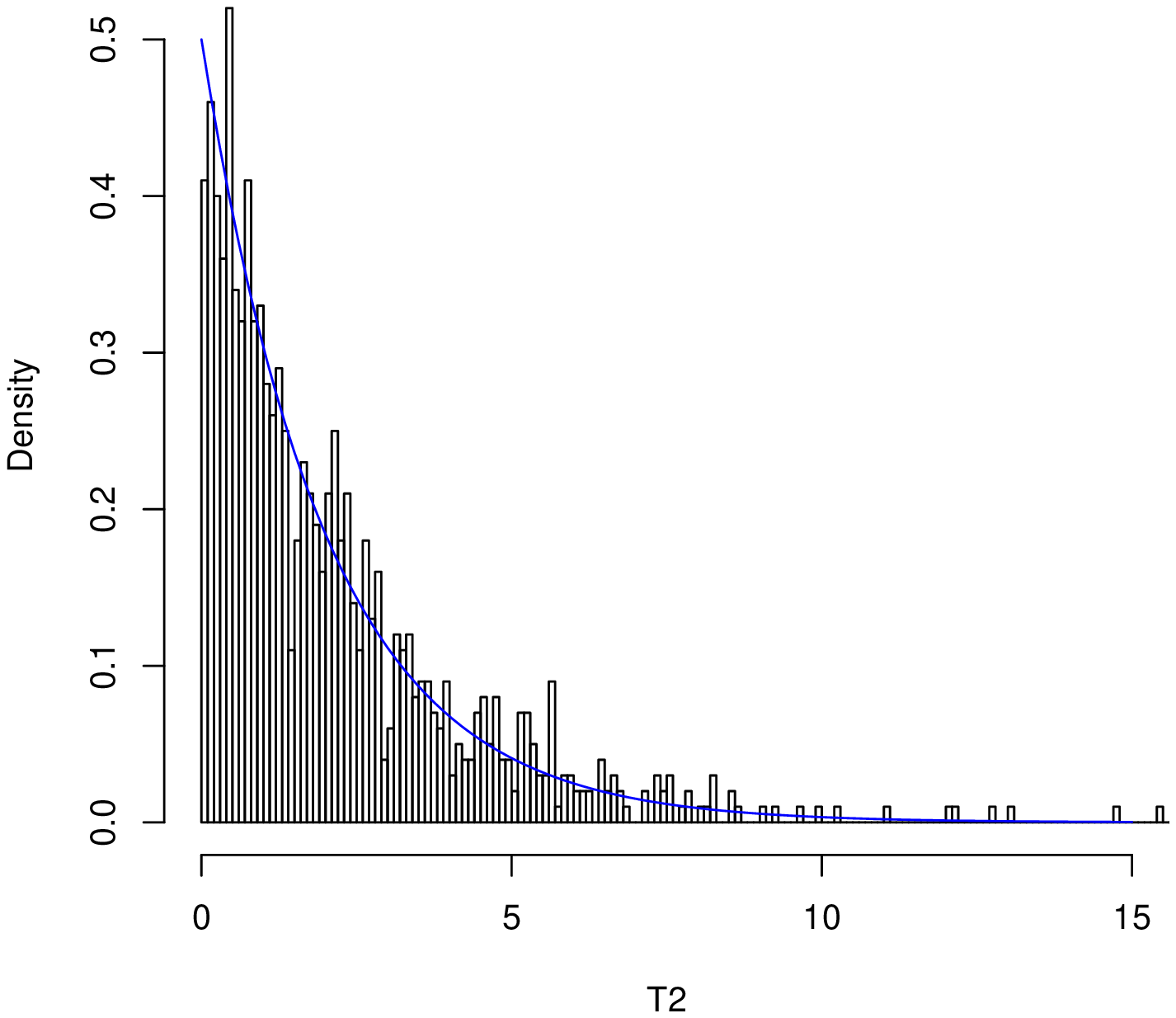}
\includegraphics[width=5.5cm]{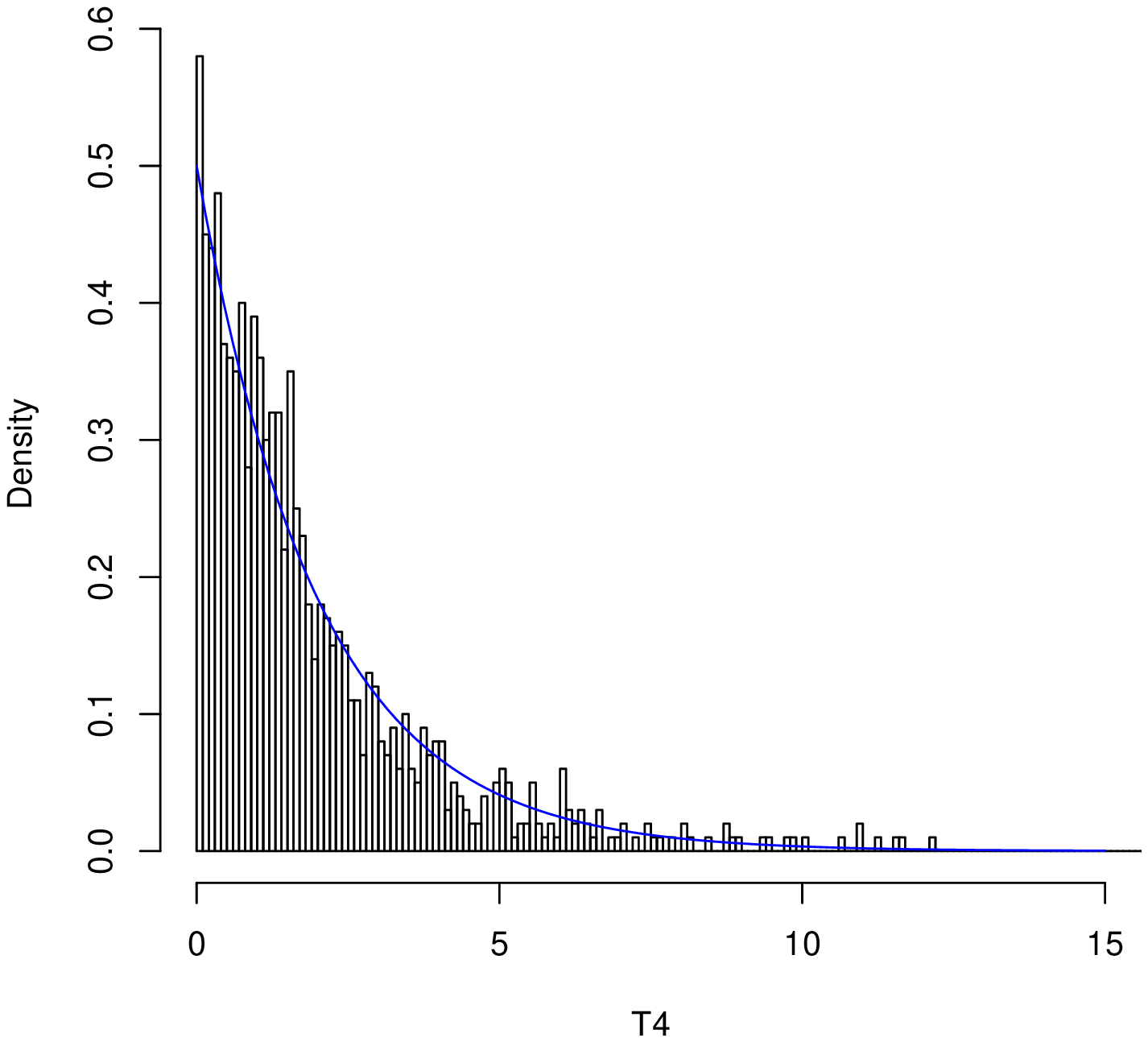}
\caption{Histogram of $T_{2d-2}$ of \eqref{eq:test-statistic-whole-line} and the density of $\chi^2(2)$ for $2d=4,6$.
}
\label{fig:whole_test}
\end{center}
\end{figure}
Figure \ref{fig:whole_4} shows for $2d=4$ the histogram of 
$p_i$ and the density of $\R{N(0,1)}$. We see that they agree with each other.
Figure \ref{fig:whole_test} shows for $2d=4,6$ the histogram of 
$T_{2d-2}$ of \eqref{eq:test-statistic-whole-line}
and the density of the chi-square distribution with $2$ d.f.  We again see a good agreement.

\section{Bivariate exponential-polynomial distribution on the positive orthant}
\label{sec:bivariate-positive}
In this section we develop holonomic gradient descent for bivariate exponential-polynomial distribution
on the positive orthant. The differential equations needed for HGD
are more difficult to derive than in the univariate case.  Also the problem of singularity of
the system of differential equations arises in the bivariate case.

Let 
\begin{align*}
h(\B{\theta},x,y) 
&= \exp \left( \sum\limits_{0 \le i+j \le d} \theta_{ij}x^i y^j \right) \\
&= \exp(\theta_{10} x + \theta_{01}y + \theta_{20} x^2 +\theta_{11} xy + \theta_{02} y^2 + \cdots + \theta_{d0}x^d + \cdots + \theta_{0d}y^d)
\end{align*}
and consider the density function
\[
f(x,y;\Btheta) = \frac{{1}}{A(\B{\theta})}h(\B{\theta},x,y) , 
\]
where 
\[
A(\B{\theta}) = \int_0^\infty \int_0^\infty h(\B{\theta},x,y) \R{d}x \R{d}y
\]
is the normalizing constant.  We call this distribution 
a bivariate exponential-polynomial distribution of degree $d$. 
Here the parameter vector $\B{\theta}$ belongs to the 
parameter space %
\begin{equation}
\label{eq:p-space}
\B{\Theta} = \{ \B{\theta} \mid A(\B{\theta}) < \infty \} .
\end{equation}
We consider the structure of $\B{\Theta}$ below in Section \ref{subsec:bivariate-parameter-space}.  
We note that
if $\B{\theta} \in \B{\Theta}$, then $h(\B{\theta},x,y)$ satisfies
\begin{align}
h(\B{\theta},x,y) &\rightarrow 0 \qquad(x \rightarrow \infty),\label{eq:proposal_x} \\
h(\B{\theta},x,y) &\rightarrow 0 \qquad(y \rightarrow \infty). \nonumber %
\end{align}

Given the sample $\B{z}=\{(x_{i},y_{i})\}_{i=1}^n$, $(1/n)$ times 
the log-likelihood function is written as
\begin{align}
\bar l(\B{\theta},\B{z})
&=  \sum\limits_{1 \le s+t \le d} \theta_{st} \overline{x^s y^t} - \log A(\B{\theta}) \label{eq:subject-2}\\
&= \theta_{10} \overline{x} + \theta_{01} \overline{y} + \cdots + \theta_{st} \overline{x^s y^t} + \cdots + \theta_{d0} \overline{x^d} + \cdots + \theta_{0d} \overline{y^d} - \log A(\B{\theta}), \nonumber 
\end{align}
{\red where $\overline{x^s y^t}=(1/n)\sum_{i=1}^n x_i^s y_i^t$.}
From \eqref{eq:subject-2} the gradient vectors is given as 
\begin{equation}
\label{eq:gradient-2}
\nabla \bar l(\B{\theta},\B{z})
=
\begin{bmatrix}
\overline{x}\\
\overline{y}\\
\vdots\\
\overline{x^s y^t}\\
\vdots\\
\overline{x^d}\\
\overline{x^{d-1}y}\\
\vdots\\
\overline{y^d}
\end{bmatrix}
-
\frac{1}{A(\B{\theta})}
\begin{bmatrix}
\partial_{10}A(\B{\theta})\\
\partial_{01}A(\B{\theta})\\
\vdots\\
\partial_{st}A(\B{\theta})\\
\vdots \\
\partial_{d0}A(\B{\theta})\\
\partial_{d-1,1}A(\B{\theta})\\
\vdots\\
\partial_{0d}A(\B{\theta})
\end{bmatrix} ,
\end{equation}
where $\partial_{ij}=\partial/\partial \theta_{ij}$. 
As in the univariate case we would like to avoid numerical integration for 
$\partial_{ij}A(\B{\theta})$, $0 \le i+j \le d$, in every step of iteration for obtaining MLE.

\subsection{Maximum likelihood estimation for the bivariate case}
We first derive differential equations satisfied by $A(\B{\theta})$.
Since there are terms like $xy$, we need to obtain different types of differential equations, which
were not needed in the univariate case.  

Let 
\begin{align}
A_x(\B{\theta}) 
&= \int_0^\infty h(\B{\theta},x,0)\R{d}x = \int_0^\infty \exp(\theta_{10}x + \theta_{20} x^2 + \cdots + \theta_{d0} x^d)\R{d}x, \label{eq:A_x} \\
A_y(\B{\theta}) 
&= \int_0^\infty h(\B{\theta},0,y)\R{d}y = \int_0^\infty \exp(\theta_{01}y + \theta_{02} y^2 + \cdots + \theta_{0d} y^d)\R{d}y \label{eq:A_y}. 
\end{align}
The values of 
\eqref{eq:A_x}, \eqref{eq:A_y} and their derivatives with respect to $\theta_{ij}$
can be obtained easily from our results for the univariate case.  Hence in the following 
derivation we treat them as known or already evaluated.

We differentiate $A(\B{\theta})$ by $\theta_{01}$ or $\theta_{10}$.  Then 
\[
\partial_{10} A(\B{\theta}) = \int_0^\infty \int_0^\infty x  h(\B{\theta},x,y) 
\R{d}x \R{d}y, \quad 
\partial_{01} A(\B{\theta}) = \int_0^\infty \int_0^\infty y h(\B{\theta},x,y) 
\R{d}x \R{d}y
\]
and we have 
\begin{equation}
\label{eq:diff-eq-3-1-2}
\partial_{10}^{s} \partial_{01}^{t} A(\B{\theta}) = \int_0^\infty \int_0^\infty x^s y^t 
h(\B{\theta},x,y) 
\R{d}x \R{d}y.
\end{equation}
On the other hand,
\begin{equation}
\label{eq:diff-eq-3-2-2}
\partial_{st} A(\B{\theta}) = \int_0^\infty \int_0^\infty x^s y^t 
h(\B{\theta},x,y) 
\R{d}x \R{d}y.
\end{equation}
From 
\eqref{eq:diff-eq-3-1-2}, \eqref{eq:diff-eq-3-2-2} we have
\[
(\partial_{st} - \partial_{10}^{s} \partial_{01}^{t}) A(\B{\theta}) = 0.
\]
Furthermore corresponding to Theorem \ref{thm:diff-eq-1}, we have the following theorem.
\begin{theorem}
$A(\B{\theta})$ satisfies the following differential equations.
\label{thm:diff-eq-1-2}
\begin{align}
\left( \sum\limits_{1 \le i+j \le d, 1 \le i } i \theta_{ij} \partial_{10}^{i-1} \partial_{01}^{j} \right) A(\B{\theta})
&= -A_y (\B{\theta}), \label{eq:diff-eq-1-2-x} \\
\left( \sum\limits_{1 \le i+j \le d, 1 \le j } j \theta_{ij} \partial_{10}^{i} \partial_{01}^{j-1} \right) A(\B{\theta})
&= -A_x (\B{\theta}). \label{eq:diff-eq-1-2-y}
\end{align}
\end{theorem}
\begin{proof}
By symmetry we only show \eqref{eq:diff-eq-1-2-x}.
We have
\begin{align}
\int_0^\infty \int_0^\infty \partial_x h(\B{\theta},x,y)\R{d}x\R{d}y
&= \int_0^\infty \left\{ \int_0^\infty \partial_x h(\B{\theta},x,y)\R{d}x \right\} \R{d}y \nonumber \\
&= \int_0^\infty \left[h(\B{\theta},x,y) \right]_{x=0}^{x=\infty} \R{d}y %
=-\int_{0}^{\infty}h(\boldsymbol{\theta},0,y)\mathrm{d}y \qquad (\text{by } \eqref{eq:proposal_x}) \nonumber \\
& =-A_y(\B{\theta}) \label{eq:diff-eq-1-2-x-1}.
\end{align}
On the other hand,
\allowdisplaybreaks
\begin{align}
\int_0^\infty \int_0^\infty \partial_x h(\B{\theta},x,y)\R{d}x\R{d}y
&= \int_0^\infty \int_0^\infty \partial_x \exp\big(\sum\limits_{0 \le i+j \le d} \theta_{ij}x^i y^j\big) \R{d}x\R{d}y \nonumber \\
&= \int_0^\infty \int_0^\infty \left( \sum\limits_{1 \le i+j \le d, 1 \le i} i \theta_{ij}x^{i-1} y^j \right) \exp(\sum\limits_{0 \le i+j \le d} \theta_{ij}x^i y^j) \R{d}x \R{d}y \nonumber \\
&= \left( \sum\limits_{1 \le i+j \le d, 1 \le i} i \theta_{ij} \partial_{10}^{i-1} \partial_{01}^j \right) \int_0^\infty \int_0^\infty h(\B{\theta},x,y) \R{d}x \R{d}y \qquad (\text{by } \eqref{eq:diff-eq-3-1-2}) \nonumber \\
&= \left( \sum\limits_{1 \le i+j \le d, 1 \le i} i \theta_{ij} \partial_{10}^{i-1} \partial_{01}^j \right) A(\B{\theta}) \label{eq:diff-eq-1-2-x-2}.
\end{align}
\eqref{eq:diff-eq-1-2-x} follows from \eqref{eq:diff-eq-1-2-x-1} and \eqref{eq:diff-eq-1-2-x-2}.
\end{proof}

In the univariate case, the important  fact was that higher-order derivatives of 
$A_d(\B{\theta})$ are written as rational function combinations of lower-order derivatives of
$A_d(\B{\theta})$.
In \eqref{eq:diff-eq-1-2-x}, \eqref{eq:diff-eq-1-2-y}, the highest order of derivatives in 
$A_d(\B{\theta})$ is $d-1$ and there are $d$  derivatives of order $d-1$:
\[
\partial_{10}^{d-1},\partial_{10}^{d-2}\partial_{01},\cdots,\partial_{10}\partial_{01}^{d-2},\partial_{01}^{d-1}.
\]
If  we want to evaluate these $d$ derivatives of order $d-1$ by solving a
system of equations, then we do not have enough equations for $d\geq 3$, because there are only two equations
 in Theorem \ref{thm:diff-eq-1-2}. We need to have more differential equations.

To obtain more equations, we operate the following set 
\[
O_{q}=\{\partial_{10}^{q},\partial_{10}^{q-1}\partial_{01},\partial_{10}^{q-2}\partial_{01}^{2},\cdots,\partial_{01}^{q}\}
\]
of $q+1$ differential operators of the same order $q$ to \eqref{eq:diff-eq-1-2-x} and \eqref{eq:diff-eq-1-2-y}.
In order to determine $q$, we count the number of differential equations obtained after operating $O_q$.

The highest order of derivatives after operating 
$O_{q}$ to \eqref{eq:diff-eq-1-2-x}, \eqref{eq:diff-eq-1-2-y} is $q+d-1$ and 
there are the following $q+d$ derivatives
\[
\partial_{10}^{q+d-1},\partial_{10}^{q+d-2}\partial_{01},\cdots,\partial_{10}\partial_{01}^{q+d-2},\partial_{01}^{q+d-1}.
\]
On the other hand there are $2(q+1)$ 
differential equations after  operating $O_{q}$ to  \eqref{eq:diff-eq-1-2-x}, \eqref{eq:diff-eq-1-2-y}.
Hence we have the right number of equations if 
we take  $q+d = 2(q+1)$ or 
\[
q = d-2.
\]
In view of 
\[
\partial_{10} A_y(\B{\theta}) 
=0, \qquad 
\partial_{01} A_x(\B{\theta}) 
=0, 
\]
when we operate 
\[
O_{d-2}=\{\partial_{10}^{d-2},\partial_{10}^{d-3}\partial_{01},\partial_{10}^{d-4}\partial_{01}^{2},\cdots,\partial_{01}^{d-2}\}
\]
to \eqref{eq:diff-eq-1-2-x}, \eqref{eq:diff-eq-1-2-y}, we have the following system of 
differential equations.
\allowdisplaybreaks
\begin{align}
\begin{bmatrix}
\partial_{10}^{d-2} & 0\\
\partial_{10}^{d-3}\partial_{01} & \vdots\\
\vdots & \vdots\\
\partial_{10}\partial_{01}^{d-3} & \vdots\\
\partial_{01}^{d-2} & 0\\
0 & \partial_{10}^{d-2}\\
\vdots & \partial_{10}^{d-3}\partial_{0\,1}\\
\vdots & \vdots\\
\vdots & \partial_{10}\partial_{01}^{d-3}\\
0 & \partial_{01}^{d-2}
\end{bmatrix}
\begin{bmatrix}
\left( \sum\limits_{1 \le i+j \le d, 1 \le i } i \theta_{ij} \partial_{10}^{i-1} \partial_{01}^{j} \right) A(\B{\theta})\\
\left( \sum\limits_{1 \le i+j \le d, 1 \le j } j \theta_{ij} \partial_{10}^{i} \partial_{01}^{j-1} \right) A(\B{\theta})\\
\end{bmatrix}
&= 
\begin{bmatrix}
\partial_{10}^{d-2} & 0\\
\partial_{10}^{d-3}\partial_{01} & \vdots\\
\vdots & \vdots\\
\partial_{10}\partial_{01}^{d-3} & \vdots\\
\partial_{01}^{d-2} & 0\\
0 & \partial_{10}^{d-2}\\
\vdots & \partial_{10}^{d-3}\partial_{01}\\
\vdots & \vdots\\
\vdots & \partial_{10}\partial_{01}^{d-3}\\
0 & \partial_{01}^{d-2}
\end{bmatrix}
\begin{bmatrix}
-A_{y}(\B{\theta})\\
-A_{x}(\B{\theta})
\end{bmatrix} %
=
-\begin{bmatrix}
0\\
\vdots\\
0\\
\partial_{01}^{d-2}A_{y}(\B{\theta})\\
\partial_{10}^{d-2}A_{x}(\B{\theta})\\
0\\
\vdots\\
0
\end{bmatrix} \label{eq:sim-equations-1}
\end{align}
We transform \eqref{eq:sim-equations-1} to a system of differential equations 
to solve for the derivatives of the highest order
\[
\partial_{10}^{2d-3},\partial_{10}^{2d-4}\partial_{01},\cdots,\partial_{10}\partial_{01}^{2d-4},\partial_{01}^{2d-3}.
\]
For any pair of non-negative integers $(s,t)$ satisfying $s+t=d-2$
let 
\begin{equation}
\label{eq:phi-psi}
\begin{split}
\phi(s,t)
&= s \partial_{10}^{s-1} \partial_{01}^t + \theta_{10} \partial_{10}^s \partial_{01}^t + \sum\limits_{2 \le i+j \le d-1, 1 \le i} i \theta_{ij} \partial_{10}^{s+i-1} \partial_{01}^{j+t} , \\
\psi(s,t)
&= t \partial_{10}^s \partial_{01}^{t-1} + \theta_{01} \partial_{10}^s \partial_{01}^t + \sum\limits_{2 \le i+j \le d-1, 1 \le j} j \theta_{ij} \partial_{10}^{s+i} \partial_{01}^{t+j-1} .
\end{split}
\end{equation}
Then \eqref{eq:sim-equations-1} is transformed to 
\begin{equation}
\begin{cases}
\sum\limits_{i+j = d ,1 \le i} i \theta_{ij} \partial_{10}^{s+i-1} \partial_{01}^{t+j}A(\B{\theta})
= - \partial_{10}^s \partial_{01}^t A_y (\B{\theta}) - \phi(s,t) A(\B{\theta}),& \\
\sum\limits_{i+j = d, 1 \le j} j \theta_{ij} \partial_{10}^{s+i} \partial_{01}^{t+j-1}A(\B{\theta})
= - \partial_{10}^s \partial_{01}^t A_x (\B{\theta}) - \psi(s,t) A(\B{\theta}). &
\end{cases}  \label{eq:sim_equations-2}\
\end{equation}
In matrix form \eqref{eq:sim_equations-2} is expressed as 
\begin{equation*}
P(\B{\theta})
\begin{bmatrix}
\partial_{10}^{2d-3}A(\B{\theta})\\
\partial_{10}^{2d-4}\partial_{01}A(\B{\theta})\\
\vdots \\
\partial_{10}\partial_{01}^{2d-4}A(\B{\theta})\\
\partial_{01}^{2d-3}A(\B{\theta})
\end{bmatrix}
=
Q(\B{\theta}), 
\end{equation*}
where 
\begin{equation}
\label{eq:co-matrix}
\begin{matrix}
P(\B{\theta})
=\begin{bmatrix}
d\theta_{d0} & \cdots & \cdots & 2\theta_{2,d-2} & \theta_{1,d-1}\\
& d\theta_{d0} & \cdots &\cdots & 2\theta_{2,d-2} & \theta_{1,d-1}\\
&&& \cdots & \cdots\\
&&& d\theta_{d0} & \cdots &\cdots & 2\theta_{2,d-2} & \theta_{1, d-1}\\
\theta_{d-1,1} & 2\theta_{d-2,2} & \cdots & \cdots & d\theta_{0d}\\
& \theta_{d-1,1} & 2\theta_{d-2,2} & \cdots & \cdots & d\theta_{0d}\\
&&& \cdots & \cdots \\
&&& \theta_{d-1,1} & 2\theta_{d-2,2} & \cdots & \cdots & d\theta_{0d} \\
\end{bmatrix}
\begin{aligned}
&\left.\begin{matrix}
\\
\\
\\
\\
\end{matrix} \right\}
\text{$(d-1)$ rows}\\    
&\left.\begin{matrix}
\\
\\
\\
\\
\end{matrix} \right\}
\text{$(d-1)$ rows}\\
\end{aligned}
\end{matrix}, 
\end{equation}
\begin{equation}
\label{eq:right-matrix}
Q(\B{\theta})
=-
\begin{bmatrix}
\phi(d-2,0)A(\B{\theta})\\
\phi(d-3,1)A(\B{\theta})\\
\vdots \\
\phi(1,d-3)A(\B{\theta})\\
\partial_{01}^{d-2}A_y(\B{\theta}) + \phi(0,d-2)A(\B{\theta})\\
\partial_{10}^{d-2}A_x(\B{\theta}) + \psi(d-2,0)A(\B{\theta})\\
\psi(d-3,1)A(\B{\theta})\\
\vdots \\
\psi(1,d-3)A(\B{\theta})\\
\psi(0,d-2)A(\B{\theta})\\
\end{bmatrix}.
\end{equation}
In $P(\B{\theta})$ empty elements in the matrix are zeros.
We give further consideration of $P(\B{\theta})$ in the next section.

If $\det P(\B{\theta}) \neq 0$, 
\begin{equation}
\label{eq:sol}
\begin{bmatrix}
\partial_{10}^{2d-3}A(\B{\theta})\\
\partial_{10}^{2d-4}\partial_{01}A(\B{\theta})\\
\vdots\\
\partial_{10}\partial_{01}^{2d-4}A(\B{\theta})\\
\partial_{01}^{2d-3}A(\B{\theta})
\end{bmatrix}
=P^{-1}(\B{\theta})Q(\B{\theta}).
\end{equation}
Hence from \eqref{eq:phi-psi}, \eqref{eq:right-matrix}, \eqref{eq:sol}
we see that 
 $\partial_{10}^{2d-3}A(\B{\theta}),\partial_{10}^{2d-4}\partial_{01}A(\B{\theta}),\cdots,\partial_{10}\partial_{01}^{2d-4}A(\B{\theta}),\partial_{01}^{2d-3}A(\B{\theta}),$
are written as rational function combinations of elements of the vector 
\begin{equation}
\label{eq:Ftheta2}
F(\B{\theta})=[
A(\B{\theta}),\partial_{10}A(\B{\theta}),\partial_{01}A(\B{\theta}),\cdots,\partial_{10}^{2d-4}A(\B{\theta}),\partial_{10}^{2d-5}\partial_{01}A(\B{\theta}),\cdots,\partial_{10}\partial_{01}^{2d-5}A(\B{\theta}),\partial_{01}^{2d-4}A(\B{\theta})]^{\mathsf{T}}.
\end{equation}
If we can evaluate $F(\B{\theta})$ at any $\B{\theta}$, then 
by \eqref{eq:gradient-2} we can obtain the maximum likelihood estimate of the
bivariate exponential-polynomial distribution.
As in the univariate case, if the initial values of $F(\B{\theta}_0)$
can be evaluated at $\B{\theta}_0$, then the value of $F(\B{\theta})$
at any other point $\B{\theta}$ can be obtained by solving the differential equation.

In the univariate case, the origin $\theta_d=0$ was the only singular point of the
differential equation \eqref{eq:theorem1}. 
In the bivariate case the set $\{\Btheta | \det P(\Btheta)=0\}$
is the set of singularities of \eqref{eq:sim-equations-1}. 
This singularity causes difficulty for HGD and in the next section we investigate 
$\det P(\Btheta)$.

{\red
\begin{remark}
\label{rem:2} 
As in Remark \ref{rem:1}, $A(\Btheta)$ satisfies 
an incomplete $A$-hypergeometric system. 
\end{remark}
}

\subsection{Evaluation of the determinant of the Pfaffian system}

We prove that $\det P(\B{\theta})$ in \eqref{eq:co-matrix} is 
given by the discriminant of a polynomial equation. %
We use the basic results on determinantal expression for  resultants and discriminants 
(cf.\ Chapter 12 of \cite{I.M.Gelfand}, Section 3.3 of \cite{H.Cohen}).
Let two polynomials $f(x)$, $g(x)$ be denoted as
\begin{align}
f(x)
&= a_m x^m + a_{m-1} x^{m-1} + \cdots + a_0 =a_m \prod_{i=1}^m(x-\alpha_i),\label{eq:pol-f}\\
g(x)
&=b_n x^n + b_{m-1} x^{m-1} + \cdots + b_0 
= b_n \prod_{i=1}^n(x-\beta_i). \nonumber %
\end{align}
The resultant $R(f,g)$ is defined as
\[
R(f,g)=a_m^n b_n^m \prod_{i=1}^m \prod_{j=1}^n (\alpha_i - \beta_j).
\]
Then the determinantal expression of $R(f,g)$ is given as follows 
((1.12) of Chapter 12 of \cite{I.M.Gelfand}, Lemma 3.3.4 of \cite{H.Cohen}).
\[
\begin{matrix}
R(f,g)=\det
\begin{bmatrix}
a_m & a_{m-1} & \cdots & \cdots & \cdots & a_0 \\
& a_n & a_{m-1} & \cdots & \cdots & \cdots & a_0 \\
&&&& \cdots & \cdots & \cdots \\
&&& a_m & a_{m-1} & \cdots & \cdots & \cdots & a_0\\
b_n & b_{n-1} & \cdots & b_0 \\
& b_n & b_{n-1} & \cdots & b_0 \\
&&&& \cdots \\
&&&&& \cdots \\
&&&&&& \cdots \\
&&&&& b_n & b_{n-1} & \cdots & b_0 
\end{bmatrix}
\begin{aligned}
&\left.\begin{matrix}
\\
\\
\\
\\
\end{matrix} \right\}
\text{$n$ rows}\\    
&\left.\begin{matrix}
\\
\\
\\
\\
\\
\\
\end{matrix} \right\}
\text{$m$ rows}\\
\end{aligned}
\end{matrix}
\]

We also consider the discriminant.
For $f(x)$ in \eqref{eq:pol-f}
the discriminant for the equation $f(x)=0$
is given by 
\[
D= (-1)^{m(m-1)/2} a_m^{2(m-1)} \prod_{1 \le i < j \le m}(\alpha_i-\alpha_j)^2.
\]
Let 
\begin{equation}
\label{eq:eq-d}
p(x;\B{\theta}) = \theta_{d0} x^d + \theta_{d-1,1} x^{d-1} + \theta_{d-2,2} x^{d-2} + \cdots + \theta_{1,d-1} x + \theta_{0d} .
\end{equation}
This polynomial will also appear in the next section in the investigation of the
parameter space $\B{\Theta}$ in \eqref{eq:p-space}.
The discriminant $D(\B{\theta})$ of the polynomial equation $p(x,\B{\theta})=0$ %
is given as ((1.29) of Chapter 12 of \cite{I.M.Gelfand}, Definition 3.3.3 of \cite{H.Cohen})
\begin{equation}
\label{eq:det-matrix}
D(\B{\theta})=\frac{1}{\theta_{d0}} R(p,p'),
\end{equation}
where 
\[
\begin{matrix}
R(p,p')=%
\det\begin{bmatrix}
\theta_{d0} & \theta_{d-1,1} & \cdots & \cdots & \cdots & \theta_{0d}\\
& \theta_{d0} & \theta_{d-1,1} & \cdots & \cdots & \cdots & \theta_{0d}\\
&&&& \cdots & \cdots &\cdots \\
&&& \theta_{d0} & \theta_{d-1,1} & \cdots & \cdots & \cdots & \theta_{0d}\\
d\theta_{d0} & \cdots & \cdots & 2\theta_{2,d-2} & \theta_{1,d-1}\\
&d\theta_{d0} & \cdots & \cdots & 2\theta_{2,d-2} & \theta_{1,d-1}\\
&&& \cdots & \cdots \\
&&&& \cdots & \cdots \\
&&&& d\theta_{d0} & \cdots & \cdots & 2\theta_{2,d-2} & \theta_{1,d-1}
\end{bmatrix}
\begin{aligned}
&\left.\begin{matrix}
\\
\\
\\
\\
\end{matrix} \right\}
\text{$(d-1)$ rows}\\    
&\left.\begin{matrix}
\\
\\
\\
\\
\\
\end{matrix} \right\}
\text{$d$ rows}\\
\end{aligned}
\end{matrix}.
\]

Using \eqref{eq:det-matrix} we give the following
theorem  on the relation of $\det P(\B{\theta})$ in 
\eqref{eq:co-matrix} and the discriminant $D(\B{\theta})$
of polynomial equation $p(x;\Btheta)=0$ in \eqref{eq:eq-d}.
\begin{theorem}
\label{thm:det-relation}
\begin{equation}
\label{eq:det-relation}
\det P(\B{\theta}) = d^{d-2} D(\B{\theta}).
\end{equation}
\end{theorem}

\begin{proof}
Define a $(2d-1)\times (2d-1)$ matrix $S$ as
\[
S=
\begin{bmatrix}
\theta_{d0} & 0 & -\theta_{d-2,2} & -2\theta_{d-3,3} & \cdots & -(d-1)\theta_{0d} & 0 & \cdots & 0\\
{\bm 0} &  &  &  & P(\B{\theta}) %
\end{bmatrix},
\]
where ${\bm 0}$ is a column vector of zeros of size $2d-2$.
Expanding the determinant with respect to the fist column we have
\begin{equation}
\label{eq:s-p-relation}
\det S = \theta_{d0} \det P(\B{\theta}).
\end{equation}
On the other hand we add the $i$-row to the $(i+d)$-th rows $(1 \le i\le d-1)$ and  then
add the $(d+1)$-st row multiplied by $d-1$ to the first row.  Then we obtain
\[
\det S
=d^{d-2}\det
\begin{bmatrix}
d\theta_{d0} & \cdots & \cdots & 2\theta_{2,d-2} & \theta_{1,d-1}\\
&d\theta_{d0} & \cdots & \cdots & 2\theta_{2,d-2} & \theta_{1,d-1}\\
&&& \cdots & \cdots \\
&&&& \cdots & \cdots \\
&&&& d\theta_{d0} & \cdots & \cdots & 2\theta_{2,d-2} & \theta_{1,d-1}\\
\theta_{d0} & \theta_{d-1,1} & \cdots & \cdots & \cdots & \theta_{0d}\\
& \theta_{d0} & \theta_{d-1,1} & \cdots & \cdots & \cdots & \theta_{0d}\\
&&&& \cdots & \cdots &\cdots \\
&&& \theta_{d0} & \theta_{d-1,1} & \cdots & \cdots & \cdots & \theta_{0d}
\end{bmatrix}.
\]
By interchanging rows 
\begin{align}
\label{eq:s-d-relation}
\det S
&= d^{d-2} (-1)^{d(d-1)} R(p,p') \nonumber \\
&= \theta_{d0} d^{d-2} \det D(\B{\theta}). \qquad(\text{by}\  \eqref{eq:det-matrix})
\end{align}
From \eqref{eq:s-p-relation}, \eqref{eq:s-d-relation} we have
\[
\det P(\B{\theta}) = d^{d-2}  D(\B{\theta}).
\]
\end{proof}

{\red
One of the reviewers asked the question of invariance of the singularities under transformations
of parameters. 
The class of holonomic functions are closed under rational transformations of arguments.
Hence if the parameters $\theta_{ij}$ are transformed by a rational transformation, the singularity
of the Pfaffian system remains to be the singularity, although the transformation itself may add its own (removable) singularity.
}

\subsection{Structure of the parameter space for the bivariate case}
\label{subsec:bivariate-parameter-space}
In this section we investigate the parameter space $\B{\Theta}$.
By the transformation
\[
x= r\cos \omega,  \quad 
y= r\sin \omega,
\]
define
\begin{equation*}
H(\Btheta,\omega) = \int_{0}^{\infty} \tilde{h}(\B{\theta},r,\omega) \R{d}r, \quad 
\tilde{h}(\B{\theta},r,\omega) = r h(\B{\theta}, r \cos \omega, r \sin \omega).
\end{equation*}
Since $\tilde h$ is non-negative, by Fubini's theorem 
$A(\B{\theta})$ is written as
\[
A(\B{\theta}) = \int_{0}^{\pi/2} H(\Btheta,\omega)\R{d}\omega.
\]
{\red 
Note that  $\log h(\Btheta,r\cos\omega, r\sin\omega)$
is a polynomial in $r$ for each $\omega\in[0,\pi/2]$ . Since 
the limit of $\log h(\Btheta,r\cos\omega, r\sin\omega)$ as $r\rightarrow\infty$
is $+\infty$ of $-\infty$, depending on the sign of the leading coefficient (i.e.\ the highest non-zero coefficient for a power) of $r$, we have
\[
 H(\Btheta,\omega)<\infty  \ \ \Leftrightarrow \ \ 
 h(\B{\theta},r\cos\omega ,r\sin\omega) \rightarrow 0 \quad (r \rightarrow \infty). %
\]
Write  $a=1/\tan\omega$.
The coefficient of the highest degree term in $y$ of $\log h(\Btheta,ay,y)$ 
is $p(a;\B{\theta})$, where $p(x;\B{\theta})$ is given in  \eqref{eq:eq-d}.
If $p(a;\B{\theta}) < 0$ for all $a\ge 0$, then $\theta_{0d}$ and the leading
coefficient of $p(a;\B{\theta})$ are negative. 
By interchanging the roles of $x$ and $y$ we also assume that the leading coefficient of
$p(a;\B{\theta})$ is  $\theta_{d0}$. 
Note that $p(a;\B{\theta}) < 0$ for all $a\ge 0$ implies $\sup_{a\ge 0}p(a;\B{\theta}) < 0$. 
Define 
\begin{equation}
\label{eq:bivariate-model-of-order-d}
\B{\Theta}'
= \{ \B{\theta} \mid  p(a;\B{\theta}) < 0,\forall a \ge  0, \ \theta_{0d}<0, \theta_{d0}<0 \}. %
\end{equation}
Then  we have   $\B{\Theta}' \subset \B{\Theta}$.
}
Note that for $\Btheta\in \B{\Theta}\setminus \B{\Theta}'$ there
exists $a> 0$ such that $p(a;\Btheta)=0$, i.e., the term of order $d$ in $y$ vanishes on the ray
$\{(ay, y), y\ge 0 \}$.
In this sense $\Btheta\in \B{\Theta}\setminus \B{\Theta}'$ may be considered as a model of order $d-1$.
We call $\B{\Theta}'$ in \eqref{eq:bivariate-model-of-order-d}
the parameter space of a {\em proper} order-$d$ model.

{\red
\begin{remark}
\label{rem:3}
One of the reviewers gave very interesting examples concerning the continuity of $H(\Btheta,\omega)$
and the relation between the convergence (finiteness) of $A(\Btheta)$ and the convergence of 
$H(\Btheta,\omega)$ for each $\omega\in [0,\pi/2]$.
\begin{itemize}
\item Let $h(\Btheta,x,y)=\exp(x^2-y^2-x)$. Then $H(\Btheta,\omega)=\infty$ for $\omega < \pi/4$, 
but $H(\Btheta, \pi/4)<\infty$.
\item Let $h(\Btheta,x,y)=\exp(-(y-x^2)^2)$. Then  $H(\Btheta,\omega)<\infty$ for each $\omega\in [0,\pi/2]$ but $A(\Btheta)=\infty$.
\item Let $h(\Btheta,x,y)=\exp(-(x-y)^2(x^2+y^2)^2)$. Then $H(\Btheta,\pi/4)=\infty$ but
$A(\Btheta)<\infty.$
\end{itemize}
These examples illustrate the difficulty in characterizing the boundary of $\B{\Theta}$.
\end{remark}
}

The above consideration gives insight into the structure of $\B{\Theta}$,
but it is still difficult to decide whether $\B{\theta}\in\B{\Theta}'$
for a given $\B{\theta}$.
We propose the following easier method for determination.
Now clearly we have
{\red
\[
p(x;\B{\theta}) < 0 \quad(\forall x \ge 0) \ 
\Leftrightarrow \ 
p(x;\B{\theta}) \text{ does not have a positive root.}
\]
}
Following the argument in \cite{N.Fukuma}, we now move
$\B{\theta}$ from an initial point in $\B{\Theta}'$, keeping
$\theta_{d0}<0,\theta_{0d}<0$, 
and consider 
when $p(x;\B{\theta})$ is no longer negative for some $x > 0$, i.e., when $p(x;\B{\theta})=0$ has a positive root. There are two cases.
\begin{enumerate}
\setlength{\itemsep}{0pt}
\item A real root moves from the negative real line to the positive real line.
\item A complex root moves to the positive real line.
\end{enumerate}
The {\red first} case corresponds to $\theta_{0d}> 0$, but this does not happen by our assumption.
Complex roots for a polynomial with real coefficients appear in conjugate pairs and
in the second case we have a multiple root on the positive real line.
Hence under the assumption  $\theta_{d0}<0,\theta_{0d}<0$, a positive root appears if and only if
the discriminant $D(\B{\theta})$ of $p(x;\B{\theta})=0$ becomes $0$ %
and the root becomes positive.  Note that $D(\Btheta)=0$ may also happen because of negative or complex multiple roots.

Based on this observation consider the complement of the hypersurface $\{\Btheta \mid D(\Btheta)=0\}$
in $\{ \B{\theta} \mid \theta_{d0} < 0, \theta_{0d} < 0 \}$:
\[
\B{\Theta}''=\{ \B{\theta} \mid \theta_{d0} < 0, \theta_{0d} < 0 \} \setminus \{\Btheta \mid D(\B{\theta})=0\}
\]
$\B{\Theta}''$ consists of disjoint open connected components (``chambers''), which we denote
by $\B{\Theta}''_i, i\in I$. Then $\B{\Theta}''$ is partitioned as
\[
\B{\Theta}'' = \bigcup_{i\in I} \B{\Theta}''_i.
\]
Note that the number of positive roots of $p(x;\Btheta)$ is constant in each chamber $\B{\Theta}_i''$.
Hence if $\B{\Theta}''_i\cap \B{\Theta}' \neq \emptyset$, then $\B{\Theta}''_i\subset \B{\Theta}'$, namely
each $\B{\Theta}''_i$ is either a subset of $\B{\Theta}'$ or disjoint from $\B{\Theta}'$.  Define
\[
I^* = \{ i\in I \mid \B{\Theta}''_i\cap \B{\Theta}' \neq \emptyset\}
=\{ i\in I \mid \B{\Theta}''_i \subset \B{\Theta}'\}.
\]
Since the hypersurface $\{\Btheta \mid D(\Btheta)=0\}$ has measure zero,  
we have the following theorem
concerning $\B{\Theta}'$ in \eqref{eq:bivariate-model-of-order-d}.

\begin{theorem}
\label{thm:p-space-thm}
Except for a set of measure zero
\begin{equation}
\label{eq:thm6.3}
\B{\Theta}' = \bigcup_{i\in I^*} \B{\Theta}''_i.
\end{equation}
\end{theorem}
Although it is difficult to completely characterize the boundaries of $\B{\Theta}''_i$'s for general $d$, 
if
the boundary between ${\red \B{\Theta}_i''}, i\in I^*,$ and ${\red \B{\Theta}_j''}, j\in I^*,$ corresponds to
negative or complex multiple roots, then
the boundary also belongs to $\B{\Theta}'$.

We illustrate the partition \eqref{eq:thm6.3} for the case of $d=3$. For any $c_1, c_2>0$, we have $p(x;\Btheta)<0, \forall x>0$ if and only if
$c_1 p(c_2 x;\Btheta)<0, \forall x>0$. This implies that we can assume $\theta_{03}=\theta_{30}=-1$
without loss of generality in  considering the partition \eqref{eq:thm6.3}.
In this case the discriminant is written as
{\red
\[
D(\Btheta)=  \theta_{12}^2 \theta_{21}^2 +4\theta_{12}^3  + 4 \theta_{21}^3 + 18\theta_{12}
\theta_{21}-27.
\]
}
On the $(\theta_{12},\theta_{21})$-plane, $D(\Btheta)=0$ consists of two curves as illustrated
in Figure \ref{fig:d3-disc}. In Figure \ref{fig:d3-disc}, chamber $A$ corresponds to
two positive roots and one negative root,  chamber $B$ corresponds to two complex roots
and one negative root, and chamber $C$ corresponds to three negative roots.
Hence the partition in \eqref{eq:thm6.3} is $B\cup C$. The boundary between $B$ and $C$
also belongs to $\B{\Theta}'$.

\begin{figure}[htbp]
\begin{center}
\includegraphics[width=7cm]{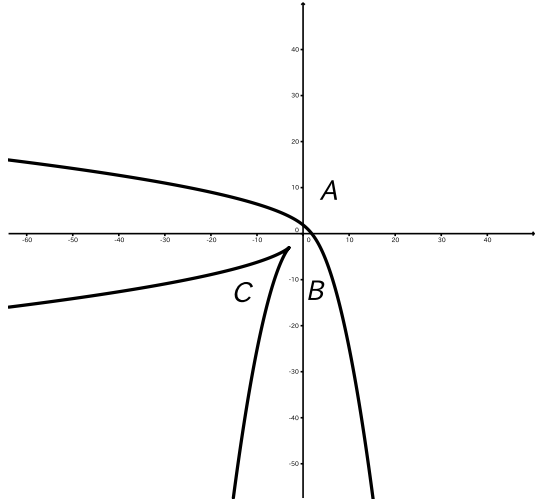}
\caption{Partition of Theorem \ref{thm:p-space-thm}
for $d=3$}
\label{fig:d3-disc}
\end{center}
\end{figure}

For maximum likelihood estimation we need to take an initial point in each chamber 
$\B{\Theta}''_i$, $i\in I^*$, of  
Theorem \ref{thm:p-space-thm}
and perform the numerical integration only for those initial points.
{\red
Note that any two points in the same chamber can be connected by a path on which $\det P(\B{\theta}) \neq 0$
and the integration of \eqref{eq:sol} does not depend on the choice of a path.
}
It is difficult to give a simple initial point $\B{\Theta}_i''$ for all 
$i \in I^*$.  For some $i \in I^*$, the following simple initial point $\B{\theta}_0$
is available. For $c_1 >0, c_2 >0$
define $\B{\theta}_0$ by 
\begin{equation*}
\theta_{d0} = -c_1, \  \theta_{0d} = -c_2, \quad  
\theta_{ij}= 0 \ \ \text{for $(i,j)\neq (0,d), (d,0)$}.
\end{equation*}
Then
\[
p(x;\B{\theta}_0)
= -c_1 x^d -c_2=0, \quad 
p'(x;\B{\theta}_0) 
= -d c_1 x^{d-1}=0
\]
do not have a common root and $R(p,p') \neq 0$. Hence %
$D(\B{\theta}_0) \neq 0$ and $\det P(\B{\theta}) \neq 0$ by
\eqref{eq:det-relation}.
Furthermore clearly $p(x;\B{\theta}_0)$ is negative for $x > 0$.
Hence $\B{\theta}_0 \in \B{\Theta}_i''$,  $i \in I^*$. 
For this $\B{\theta}_0$ the normalizing constant and its derivatives are easily evaluated as
\begin{align*}
\partial_{ij}A(\B{\theta}_0)
&= \int_0^\infty \int_0^\infty x^i y^j \exp(-c_1 x^d -c_2 y^d )\R{d}x\R{d}y \\
&= \int_0^\infty x^i \exp(-c_1 x^d) \R{d}x  \int_0^\infty y^j \exp(-c_2 y^d) \R{d}y\\
&= \frac{1}{d} c_{1}^{-(i+1)/d}\Gamma \left( \frac{i+1}{d} \right) \frac{1}{d}c_2^{-(j+1)/d}\Gamma \left( \frac{j+1}{d} \right).
\end{align*}

Although we do not show numerical results for the bivariate case, for $d=2$ the computation of the
normalizing constant and MLE is fast and the asymptotic distribution of MLE has been checked.
For $d=3$, the computation of the normalizing constant is fast, but the computation of MLE is
somewhat heavy at current implementation in MATLAB.  This seems to be due to high dimensionality 
(9 parameters) of the model for $d=3$.

\section{Some discussions}
\label{sec:discussion}

In this paper we discussed the maximum likelihood estimation of the exponential-polynomial distribution.
Here we discuss some possible extensions of the distribution and topics for further research.

In the exponential-polynomial distribution we have a polynomial as the exponent of the
exponential function.  We can add another polynomial to the exponential-polynomial distribution,
if this polynomial is non-negative over the sample space.  
Recall that the problem concerning
non-negative polynomials was also essential for understanding the structure of the parameter space
for  the bivariate exponential-polynomial distribution in Section \ref{subsec:bivariate-parameter-space}.
Let \[
p(x;\Beta)=\eta_0 + \eta_1 x + \dots + \eta_h x^h
\]
be a polynomial in $x$. Consider the following density on the positive real line:
\begin{align*}
f(x;  \Beta, \Btheta) &= \frac{1}{\tilde A(\Beta,\Btheta)} p(x;\Beta)
\exp(\theta_1 x + \dots + \theta_d x^d), \\
\tilde A(\Beta,\Btheta)&= \int_0^\infty p(x;\Beta)
\exp(\theta_1 x + \dots + \theta_d x^d) \R{d}x.
\end{align*}
The normalizing constant $\tilde A(\Beta,\Btheta)$ can be evaluated as
\[
\tilde A(\Beta,\Btheta)=\sum_{i=0}^h \eta_i \int_0^\infty x^i \exp(\theta_1 x + \dots + \theta_d x^d)\\
=\sum_{i=0}^h \eta_i \partial_1^i A_d(\Btheta),
\]
where $A_d(\Btheta)$ is given in \eqref{eq:norm-constant}.  Hence from the view point of holonomic gradient
descent this generalization can be easily handled.  However, in the estimation of this density
we need to guarantee that $p(x;\hat\Beta)$ is a non-negative polynomial for $x\ge 0$.
This problem was considered in Fushiki et al.\ (\cite{T.Fushiki}).  
They  showed that the maximum likelihood estimation under the restriction of non-negativity of $p(x;\hat\Beta)$
can be performed with the technique of semidefinite programming.
We can also use the parameterization of non-negative polynomials given in Proposition 3.3 of \cite{N.Kato}.  
See also Section 9, Chapter V of \cite{S.Karlin}.

For the univariate case we derived score tests for determining the order $d$ of the
model.  The difficulty in model selection is  the fact 
that the model of order $d-1$ is on the boundary of the model of order $d$.
In this paper we did not discuss the problem of model selection for the bivariate case, 
because the boundary  is much more difficult compared to the univariate case, as discussed in Section 
\ref{subsec:bivariate-parameter-space}.  Also in the
bivariate case, as the model of order $d$ we included all monomials $x^d, x^{d-1}y, \dots, y^d$ of order $d$.  However we may omit some monomials among these $d+1$ monomials.  
The structure of the boundary of the model seems to depend 
on the choice of monomials of order $d$.
Model selection procedures for the
bivariate case is left to a future study.

{\red
It is of interest to generalize our results for bivariate case to higher dimensions.
As remarked in Remarks \ref{rem:1} and \ref{rem:2} we can use general theory of
$A$-hypergeometric systems to obtain results for the exponential-polynomial distribution in general dimension.
In the bivariate case the singularity of the Pfaffian system is described
in terms of the discriminant $D(\B{\theta})$ in Theorem 6.3.
It is of interest to generalize this result to higher dimensions.
}

\bigskip
\noindent
{\bf Acknowledgment} \ \ 
The authors are very grateful to Satoshi Kuriki, Nobuki Takayama, Tamio Koyama and two reviewers  for very useful suggestions.

\bibliographystyle{abbrv}
\bibliography{exp-pol}

\end{document}